\documentclass[12pt]{article}
\scrollmode
\usepackage{mathrsfs,color}
\usepackage{amsfonts}
\usepackage{amsthm}
\usepackage{amsmath}
\usepackage{amssymb}
\usepackage{graphicx}

\allowdisplaybreaks

\def\eps{\varepsilon}

\def\qed{\hfill\rule{.2cm}{.2cm}}
\def\P{{\mathbb P}}

\def\Z{{\mathbb Z}}

\def\R{{\mathbb R}}
\def\Q{{\mathbb Q}}

\def\1{{\mathbf 1}}
\def\T{{\mathbf T}}

\def\o{\omega}  
\def\l{\lambda} \def\hl{\hat \l}

\def\div{ \mathrm{div} }

\def\bB{\bar B}

\newcommand{\FF}          {\mathcal{F}}

\newcommand{\LL}         {\mathcal{L}}

\newcommand{\HH}         {\mathcal{H}}
\newcommand{\WW}         {\mathcal{W}}

\def\O{\Omega}

\newtheorem{theo}{Theorem}[section]
\newtheorem{prop}[theo]{Proposition}
\newtheorem{lm}[theo]{Lemma}
\newtheorem{cor}[theo]{Corollary}
\newtheorem{rmk}[theo]{Remark}
\newtheorem{df}[theo]{Definition}
\def\beq{\begin{equation}}
\def\eeq{\end{equation}}
\newcommand{\bei}{\begin{itemize}}
\newcommand{\eei}{\end{itemize}}
\newcommand{\ben}{\begin{enumerate}}
\newcommand{\een}{\end{enumerate}}
\newcommand{\beqn}{\begin{eqnarray}}
\newcommand{\beqnn}{\begin{eqnarray*}}
\newcommand{\eeqn}{\end{eqnarray}}
\newcommand{\eeqnn}{\end{eqnarray*}}
\newcommand{\brm}{\begin{rmk}}
\newcommand{\erm}{\end{rmk}}

\newcommand {\tr}[1] {{#1}\cdot }

\definecolor{darkred}{rgb}{0.9,0.1,0.1}

\title{Steady states, fluctuation-dissipation theorems and homogenization for reversible diffusions in a random environment }
\author{
P.~Mathieu \setcounter{footnote}{6}\footnote{ Aix-Marseille Universit\'e, CNRS, Centrale Marseille, I2M UMR 7373, 13453 Marseille, FRANCE.
{\sl pierre.mathieu@univ-amu.fr}}
\and
A.~Piatnitski \setcounter{footnote}{3}\footnote{The Arctic University of Norway, campus Narvik, P.O.Box 385, 8505, NORWAY, and the Institute for Information Transmission Problems  of RAS, Moscow, 127051, Bolshoy Karetny per., 19,   RUSSIA.
{\sl apiatni@iitp.ru}}
}

\begin{document}
\maketitle

\begin{abstract}
Prolongating our previous paper on the Einstein relation, we study the motion of a particle diffusing in a random reversible environment  when subject to a small external forcing.
In order to describe the long time behavior of the particle, we introduce the notions of steady state and weak steady state. We establish the continuity of weak steady states for an ergodic and uniformly elliptic environment.
When the environment has finite range of dependence, we prove the existence of the steady state and weak steady state and compute its derivative at a vanishing force. Thus we obtain a complete 'Fluctuation-Dissipation Theorem' in this context as well as the continuity of the effective variance.
\end{abstract}

\section{Introduction}\label{s_iintro}

Prolongating the work started in \cite{kn:GMP}, we study the motion of a particle diffusing in a random reversible environment  when subject to a small external forcing. The external force we consider is a constant in time vector field in some direction $e_1$ and strength $\lambda$. We think of $\lambda$  as being small.

Long time properties of the motion of our particle depend on the process of the environment seen from the particle: in the absence of the external force, the process of the environment seen from the particle is at equilibrium and
the motion of the diffusing particle is diffusive (obeys the central limit theorem). When a constant external force is added, the process of the environment seen from the particle is off equilibrium and the motion of the particle becomes ballistic.
In order to get a law of large numbers, one has to study appropriate invariant measures for the environment seen from the particle; we call such measures 'steady states'.
Although the existence of a steady state was proved for environments with a finite range of correlation in \cite{kn:KoKr06}, nothing was known until recently about the way it depends on $\lambda$. A first partial answer was given in \cite{kn:GMP} where we computed the derivative of the effective velocity and thus obtained the so-called Einstein relation. In the present paper, we shall investigate regularity properties of the steady state itself.

This question is of general interest in physics where studying the response of a system to a small perturbation is often a fruitful experimental procedure. A first example of such a situation is the work of Perrin on the Brownian motion of minute particles suspended in liquids, see \cite{kn:Perrin},  that confirmed the theoretical predictions of Einstein about Brownian motion and the existence of atoms, see \cite{kn:Einstein}. Another well-known example is the Green-Kubo relation expressing transport coefficients in terms of correlations, see \cite{kn:Ku85}. Such results are usually referred to as Fluctuation-Dissipation theorems or Linear Response theory in the physics literature. We refer to \cite{kn:HM} and their references for applications to climate change among others.

Reversible diffusions in a random environment are also an example of models obeying homogenization (\cite{kn:DFGW}, \cite{kn:JKO}, \cite{kn:KV}, \cite{kn:Ko80}, \cite{kn:Ko85}, \cite{kn:Osada}, \cite{kn:PV} among others). Studying the effect of imposing a small drift in the equation is then a way to test the robustness of homogenization properties. Indeed our result on the continuity of the steady state rely on our ability to obtain bounds on the effect of the external forcing that are uniform in time, see in particular Lemma \ref{lm:estimateAlambda}.
Let us also mention that similar issues are currently addressed in the context of deterministic dynamical systems, see \cite{kn:Baladi} and references therein.

\medskip

Let $\Omega$ be the space of smooth $d\times d$ symmetric non-negative matrix functions defined on $\mathbb R^d$.
We equip this space with the topology of uniform convergence on compact subsets of $\mathbb R^d$. We let
$\mathbb R^d$ act on $\Omega$ by additive translations. We denote this action by  $x.\o$.

Let $\mathbb Q$ be a Borel probability measure on $\Omega$.

\noindent
{\bf Assumption 1}.   The action $(x,\omega)\mapsto x.\, \omega.$ preserves the measure $\mathbb Q$  and  is ergodic.

\medskip
We first introduce the diffusion process without external forcing. Let  $(X^\o_0(t)\,;\,t\ge 0)$ be the solution of the stochastic differential equation in $\R^d$:
\beqn\label{eq:introsde}
dX^\o_0(t)=b^\o(X^\o_0(t))dt  +\sigma^\o(X^\o_0(t))dW_t\,;\, \qquad X^\o_0(0)=0\,,
\eeqn
where $\sigma^\o(x)=\sigma(x.\o)$ is a stationary  $d\times d$ matrix,  $b^\o(x)=\frac 12\div (\sigma^\o(x)(\sigma^\o)^*(x))$, and $(W_t\,;\, t\ge 0)$ is a $d$-dimensional Brownian motion defined on some probability space $(\WW,\FF,P)$. In the sequel, we use the notation $a^\omega(x)=\sigma^\o(x)(\sigma^\o)^*(x)$ and $a(\omega)=\sigma(\o)(\sigma)^*(\o)$.  The vector field $b^\omega$ is stationary therefore of the form
$b^\omega(x)=b(x.\omega)$ for some vector valued function $b$ defined on $\Omega$.

\medskip
Our goal is to study the behaviour of  the diffusion process $X^\o_0(t)$  perturbed by a fixed small force. The corresponding equation for the perturbed process reads
\beqn\label{eq:introsde_perturb}
dX^{\lambda,\o}_0(t)=b^\o(X^{\lambda,\o}_0(t))dt + \lambda a^\o(X^{\lambda,\o}_0(t)) e_1\,dt  +\sigma^\o(X^{\lambda,\o}_0(t))dW_t\,;\,\ \ 
X^{\lambda,\o}_0(0)=0\,,
\eeqn
where  $e_1$ is a fixed vector in $\mathbb R^d$, and $\lambda\in\mathbb R$.

In the paper we assume that  the diffusion coefficient in \eqref{eq:introsde}, \eqref{eq:introsde_perturb}
satisfies  the following uniform ellipticity condition:

\noindent
{\bf Assumption 2}.  There is  $\varkappa>0$ such that the following estimates hold:
$$
\varkappa |\zeta|^2\leq |\sigma(\omega)\zeta|^2\leq \varkappa^{-1} |\zeta|^2, \quad \hbox{for all }\omega\in\Omega\hbox{ and } \zeta\in\mathbb R^d.
$$

\noindent
We also assume that  the diffusion coefficient in \eqref{eq:introsde}, \eqref{eq:introsde_perturb} has smooth realizations:

\noindent
{\bf Assumption 3:} for any environment $\o$, the function
$x\rightarrow \sigma^\o(x)$ is smooth.

\medskip

The asymptotic behaviour of the non-perturbed symmetric diffusion \eqref{eq:introsde} was widely studied in the 70's and 80's.
It was proved, see \cite{kn:KV}, \cite{kn:Ko80}, \cite{kn:Ko85}, \cite{kn:Osada},  \cite{kn:PV}, that, under general ergodicity assumptions,
the process $X^\o_0$ shows a diffusive behaviour and satisfies the invariance principle. We endow the path space with the topology of locally uniform convergence. Then the law of the family of rescaled processes
$(\eps X^\o_0(t/\eps^2)\,;\, t\ge 0)$ weakly converges towards the law of a Brownian motion with some
covariance matrix $\Sigma$.


If $\lambda>0$ then the process $X^{\lambda,\o}_0$ is ballistic.  It was shown in \cite{kn:GMP} that
it satisfies the quenched estimates
 $$
 c_1\lambda t\leq E (X^{\lambda,\o}_0(t)\cdot e_1)\leq c_2\lambda t
 $$
with deterministic constants $c_1$, $c_2$, $0<c_1<c_2$ that only depend on the ellipticity constants
and the dimension
and do not depend on $\lambda$; here the symbol $E$ stands for the expectation related to the measure $P$ on $(\mathcal{W},\mathcal{F})$. We generalize this estimate in Lemma \ref{lm:estimateAlambda}.

However, these estimates do not  automatically imply the law of large numbers (LLN).  The LLN was proved in \cite{kn:LS} under the condition that the diffusion matrix $a^\o(x)$ has a finite range of dependence, see {\bf Assumption 4} below. The proof is based on the construction of regeneration times. This technique also yields the central limit theorem for $X^{\lambda,\o}_0$; we call $\Sigma_\lambda$ the asymptotic variance. 

These results can be better understood using the {\sl point of view of the particle} introduced in \cite{kn:PV}.
Define the process $\o^0(t)=X_0^\o.\o$, respectively  $\o^\lambda(t)=X_0^{\lambda,\o}.\o$.
One checks that $\omega^0(.)$ and $\omega^\l(.)$  are  Markov processes, and that $\Q$ is a reversible invariant measure of $\omega^0(.)$.  Using the Dirichlet form of $\omega^0(.)$, we define the Sobolev space $H^1(\Omega)$
and its adjoint $H^{-1}(\Omega)$. It was shown in \cite{kn:KV}, \cite{kn:DFGW} that the invariance principle holds for additive functionals of elements of $H^{-1}(\Omega)$. The invariance principle stated above for the
process $X^\omega_0$ is
a consequence of these more general results.

For positive $\lambda$ the measure $\mathbb Q$ is not invariant any more. Following \cite{kn:KoKr04}
we use the notion of {\sl steady state}:

\begin{df} \label{df:steadystateIntro}
Let $\lambda>0$. A  Borel probability measure $\nu_\lambda$ on $\Omega$ is called steady state if
for any bounded local function $f$, for $\mathbb Q$ almost all $\omega$ and $P$ almost surely we have
$$
\lim\limits_{t\to\infty}\ \frac{1}{t}\int_0^t f(\omega^\lambda(s))\,ds =\nu_\lambda(f),
$$
where $\omega^\lambda(s)=X_0^{\lambda,\omega}(s).\omega$.
\end{df}

Note that, if it exists, the steady state is an invariant measure for the Markov process  $\omega^\lambda(.)$ and it is  unique.

The existence of the steady state is proved in \cite{kn:KoKr06} for a model of a diffusion in a random
environment that differs a bit from ours and satisfies {\bf Assumption 4} below on a finite range of dependence.  In Section \ref{s_sec2} we shall also obtain
the existence of $\nu_\lambda$ assuming finite range of dependence  by a method that is more explicit than in \cite{kn:KoKr06}.
The existence of the steady state is not known for a general stationary ergodic environment. Furthermore, even if we happened to know that it exists for all $\lambda$,  it would not directly follow from the definition whether $\nu_\lambda$ converges to $\mathbb Q$
as $\lambda\to0$.

This motivates us to modify the definition of a steady state and to introduce the notion of weak steady state in the definition below.

The weak steady state is defined on a special subset of the space $H^{-1}(\Omega)$
that we call ${\tilde H}^{-1}_\infty(\Omega)$. The precise definition will be given in Section \ref{s_sec1}. Loosely speaking, one may think of elements in $H^{-1}(\Omega)$ as
function $f$ on $\Omega$ that can be written as the divergence of some stationary, square integrable vector field, say $f=\div F$. We call $H^{-1}_\infty(\Omega)$ the set
of $f$ in $H^{-1}(\Omega)$ for which we can choose a bounded $F$. Note that $H^{-1}_\infty(\Omega)$ is naturally endowed with a Banach space structure.
We further let  ${\tilde H}^{-1}_\infty(\Omega)$ denote the closure in $H^{-1}_\infty(\Omega)$ of the linear set of $f$ in $H^{-1}_\infty(\Omega)$ for which we can choose a bounded and local $F$. Precise definitions are given
at the beginning of Section \ref{s_sec1}.

A typical example of an element of $H^{-1}_\infty(\Omega)$ is obtained choosing $F(\omega)=a(\omega)$. Then $f(\omega)=2 b(\omega)$ is the drift term in equation
(\ref{eq:introsde}).

We shall see that, although an element of $f\in H^{-1}_\infty(\Omega)$ need not be a function, it still makes sense to consider the additive functional
$$A^{\lambda,\omega}_{0,f}(t)=\int_0^t f(\omega^\lambda(s))\,ds.$$

We thus define the notion of

\begin{df} \label{df:weaksteadystateIntro}
Let $\lambda>0$. A  continuous linear functional $\nu_\lambda$ on ${\tilde H}^{-1}_\infty(\Omega)$ is called weak steady state if
for any  $f$ in ${\tilde H}^{-1}_\infty(\Omega)$, then
\beqn\label{eq:convH-1}
\lim\limits_{t\to\infty}\ \frac{1}{t}A^{\lambda,\omega}_{0,f}(t) =\nu_\lambda(f),
\eeqn
in $L^1(\WW, P)$ for $\mathbb Q$ almost all $\omega$.
\end{df}

As we shall see in Section \ref{s_sec1}, if the convergence in (\ref{eq:convH-1}) holds for any $f$ in ${\tilde H}^{-1}_\infty(\Omega)$,
then the limit is automatically a linear continuous functional on ${\tilde H}^{-1}_\infty(\Omega)$.

Observe that due to Lemma \ref{lm:estimateAlambda} below, we could replace in Definition \ref{df:weaksteadystateIntro}
 the convergence in  $L^1(\WW, P)$ with the convergence
in $L^p(\WW, P)$ for any $p\geq 1$.  Also, due to the same Lemma, if  $ \frac{1}{t}A^{\lambda,\omega}_{0,f}(t)$ converges
$P$ almost surely, then the convergence holds in $L^1(\WW, P)$  as well.

We prove the Lipschitz continuity of weak steady states:

\begin{theo} \label{theo:continuityIntro}
There exists a constant ${\tt C}_1$ satisfying the following:
if for $\lambda$ with $0\leq \lambda\leq 1$ and $f$ in $H^{-1}_\infty(\Omega)$ the limit
$$
\lim\limits_{t\to\infty}\ \frac{1}{t}A^{\lambda,\omega}_{0,f}(t) :=\nu_\lambda(f),
$$
exists in $L^1(\WW, P)$ for $\mathbb Q$ almost all $\omega$, then
\beqn\label{eq:continuityestimateH-1} \vert \nu_\lambda(f)\vert\le{\tt C}_1\lambda \Vert f\Vert_{H^{-1}_\infty(\Omega)}.\eeqn
In particular, if the weak steady state exists for all $\lambda\in[0,1]$, then  $\nu_\lambda(f)$ converges to
$0$, as $\lambda\to0$ for all $f\in{\tilde H}^{-1}_\infty(\Omega)$.
\end{theo}

\begin{rmk}\label{r_const}
In the next section we introduce the space $H^1(\Omega)$ in such a way that these functions have zero mean value. With this definition
the duality between functions from $H^{-1}(\Omega)$ and constants does not make sense. However,  under our assumptions, the generators $D_j$, $j=1,\ldots, d$ of the action $x.\omega$ are such that $\sqrt{-1}D_j$ are self-adjoint in $L^2(\Omega)$.
Therefore, for any $F=(F_1,\ldots, F_j)$ such that $F_j$ belongs to the domain of  $D_j$ we have
$\int_\Omega\mathrm{div}F(\omega)d\mathbb Q=-\int_\Omega F\cdot \nabla(1)d\mathbb Q=0$.

Thus all elements of $H^{-1}(\Omega)$ are centered in a certain sense.
In particular, $\int_\Omega d\mathbb Q(\omega) E[A^{0,\omega}_{0,f}(t)]=0$ for all $t$.

Therefore Equation (\ref{eq:continuityestimateH-1}) does indeed express
the Lipschitz continuity of the weak steady state $\nu_\lambda$, considered as a linear functional on ${\tilde H}^{-1}_\infty(\Omega)$.

\end{rmk}

In Section \ref{s_sec2}, we prove that weak steady states exist for all $\lambda$ if $\mathbb Q$ has finite range of dependence,
see {\bf Assumption 4} below.

\medskip

From now on, we shall discuss properties of diffusions in a media satisfying the following finite range of dependence property:
for a Borel subset $F\subset\R^d$, let $\HH_F$ denote the  $\sigma$-field
generated by $\{\sigma(x.\o)\,:\, x\in F\}$.
We assume that:

\noindent
{\bf Assumption 4:}  there exists $R$ such that for any Borel subsets $F$ and $G$
such that $d(F,G)>R$ (where $d(F,G)=\inf\{\vert x-y\vert\,:\, x\in F, y\in G\}$ is the distance between $F$ and $G$) then
\beqn \label{eq:indep0} \HH_F \hbox{ \rm and } \HH_{G} \hbox{ are independent}\,.
\eeqn

As already mentioned, under {\bf Assumption 4}, then steady states and weak steady states exist for all $\lambda$ and Theorem \ref{theo:continuityIntro} applies. We can go
one step further and show that $\nu_\lambda(f)$ has a derivative at $\lambda=0$. This is the
content of the next Theorem.

\begin{theo}\label{t_ssslIntro}
Let $f$ belong to  ${\tilde H}^{-1}_\infty(\Omega)$. Then,
the derivative of $\nu_\lambda(f)$ at $\lambda=0$ exists.
\end{theo}

Our main tool for proving the existence of the steady state and Theorem \ref{t_ssslIntro} are regeneration times.
As a matter of fact, regeneration times were already the main tools in \cite{kn:LS} (for the proof of the law of large
numbers and c.l.t. for $X_0^{\lambda,\omega}$) and in \cite{kn:KoKr06} to establish the existence of steady states;
see also \cite{kn:kesten} \cite{kn:SZ} for random walks.

In order to prove Theorem \ref{t_ssslIntro}, one needs regeneration times that do not explode faster than $\lambda^{-2}$
as $\lambda$ tends to $0$. We already faced this issue in  \cite{kn:GMP} and there we introduced appropriate modifications
to the definitions in \cite{kn:LS} to achieve the right order of magnitude. The construction we shall use here
differs a bit from \cite{kn:GMP} but it also provides
regeneration times of order $\lambda^{-2}$.
The other key ingredient in the proof of Theorem \ref{t_ssslIntro} is an explicit
expression of $\nu_\lambda(f)$ in terms of regeneration times.
Our definition makes the regeneration time depend on the function $f$.

\medskip

The proof of Theorem \ref{t_ssslIntro} also gives the value of the derivative.
Let us denote by  $\bar\Gamma(f)$ the derivative of $\nu_\lambda(f)$ at $\lambda=0$ as in Theorem \ref{t_ssslIntro}.
We now give various interpretations of $\bar\Gamma(f)$.

One proof of the invariance principle is based on the existence of a {\it corrector}:
let $\LL^\o$ be the generator of the process $X^\o_0$.
The corrector is  a (random)
function $\chi$ defined on $\R^d$, with values in $\R^d$ and satisfying the equation
\beqn\label{eq:introcorrector}
\LL^\o\chi=-b^\omega\,.
\eeqn

One shows that equation (\ref{eq:introcorrector}) has a solution with a stationary gradient, see Section \ref{sss_corr}.
If  $\sigma(\cdot)$ has finite range of dependence and $d\geq 3$, then as was proved in
\cite{kn:GMa} and \cite{kn:GO},  equation \eqref{eq:introcorrector} has a stationary solution. We show that
\beqn\label{ggamma}
\bar\Gamma(f)=-2\int_\O \chi(\o)f(\o)\,d\mathbb Q.
\eeqn
Notice that in general  equation \eqref{eq:introcorrector} need not have a stationary solution. However, the interpretation
of  the derivative of the steady state at $\lambda=0$ as the corrector remains valid in a weaker form,
see Proposition \ref{p_gamma}.

In Lemma \ref{l_cova} we give another interpretation of $\bar\Gamma(f)$ as a covariance. Finally we can also obtain
$\bar\Gamma(f)$ as a drift term for the scaling limit of a perturbed diffusion with vanishing strength in the
so-called Lebowitz-Rost scaling discussed in Section \ref{s_leb-rost}.

These last interpretations of $\bar\Gamma(f)$ are in good agreement with Fluctuation-Dissipation-Theorems that
predict that the linear response of a system in equilibrium can be expressed as a correlation.

 In the Appendix A, we briefly discuss the case of a periodic environment where the construction of steady states
is immediate and the expression of the derivative of the steady state in terms of the corrector (\ref{ggamma})
follows by directly
comparing the periodic boundary value problems for PDE's satisfied by these quantities.



The proof of Theorem \ref{t_ssslIntro} relies of the Continuity Lemma \ref{t_fop} which gives
the scaling limit  on the regeneration scale of the joint law of  $X^{\lambda,\omega}_0$ and $A^{\lambda,\omega}_{0,f}$
for a local function $f$ in ${\tilde H}^{-1}_\infty$.

Another important consequence of Lemma \ref{t_fop} is the continuity of the asymptotic variance
$\Sigma_\lambda$ at $\lambda=0$.


\medskip





\medskip

The organization of the paper is as follows.

In Section \ref{s_LR-scal} we consider rather general stationary environments and discuss scaling limits of additive functionals of the environment seen from the particle either in the case $\l=0$
or, more generally, in the Lebowitz-Rost scaling. The material from this part cannot be called `new': it is mainly a rephrasing of arguments borrowed from references \cite{kn:KV}, \cite{kn:DFGW}
and \cite{kn:LR}. For the background materials we refer to the books \cite{kn:JKO} and \cite{kn:KoLaOl12}. However we found it necessary to include some details in this part as the precise statements needed in the sequel are not always easy to find in the references.
We believe it makes the paper more self-contained and easier to read.

In Section \ref{s_sec1} we investigate continuity properties of steady states and prove Theorem \ref{theo:continuityIntro}.

Section \ref{s_sec2} is devoted to the construction of regeneration times and of the steady state and weak steady state assuming the environment has finite range of dependence.
Our regeneration times are not exactly as in \cite{kn:KoKr04,kn:KoKr06}.
Indeed, in our construction, the definition of the regeneration times depends on the function $f$. This point of view allows for an explicit expression of $\nu_\l(f)$.


In Section \ref{ss_h-1} we let $\l$ tend to $0$.
The crucial role here is played by the  estimates obtained in \cite{kn:GMP} and  by  uniform estimates for the scaled regeneration times
in the case $f\in {\tilde H}^{-1}_\infty$.
We obtain the general Continuity Lemma \ref{t_fop}.
As a first consequence we prove the existence of and identify the derivative of $\nu_\l$ at $\l=0$. Finally, in Section \ref{s_sec4} we also obtain a continuity property of the asymptotic variance
$\Sigma_\l$  and we derive from the general continuity lemma the validity of the Einstein relation in a way that differs from \cite{kn:GMP}.

\medskip
\begin{rmk}

The questions addressed in this paper can also be raised for discrete models of random walks among random conductances.
This is the object of the recent paper \cite{kn:GGN}. (Our two papers are simultaneous. They cannot be called 'independent'
as the two teams kept contacts during all the elaboration of the two preprints.)

In \cite{kn:GGN}, the authors consider random walks with uniformly elliptic conductances, only the i.i.d. case being studied. Their main result is the Einstein relation, which
they obtain following a strategy similar to that in \cite{kn:GMP}. In particular they construct regeneration times of the correct order.
On top of it \cite{kn:GGN} discusses regularity properties of the steady state $\nu_\lambda$.

The approach used in \cite{kn:GGN} is more quantitative than ours: the authors
assume that $d\geq 3$, so that there exists a stationary corrector and local bounded functions are in $H^{-1}(\Omega)$. Furthermore, they
crucially rely on results from \cite{kn:Mourrat} that quantify the ergodicity of the environment seen from the particle.
As a result, they obtain the continuity of the steady state acting on local bounded continuous functions - that we do not get here - and they show, for $d\geq 3$, fluctuation-dissipation relations
similar to our Theorem \ref{t_sssl} and Corollary \ref{fdt}.

Here we preferred to take the '$H^{-1}$ point of view' as a starting point: we view the steady state as a linear functional on ${\tilde H}^{-1}_\infty$ rather than as a measure,
see Definitions \ref{df:steadystateIntro} and  \ref{df:weaksteadystateIntro}. This allows us to include the two-dimensional case and
to get continuity results for general ergodic environments, see Theorem \ref{theo:continuityIntro}.
As for the FDT, we do not use quantitative bounds on the ergodicity of
environment seen from the particle but rather make an extensive use of scaling limits, see Lemma \ref{t_fop}.
Our approach also yields the
results on the continuity of the variance.

\end{rmk}


\section{Homogenization of additive functionals}\label{s_LR-scal}

Let $\Omega$ be a separable topological space, equipped with a measurable action of  $\mathbb R^d$ that we denote
$$
(x,\omega)\mapsto x.\, \omega.
$$
Let $\mathbb Q$ be a Borel probability on $\Omega$.
We denote by ${D}= ({D}_1,\ldots,{D}_d)$ the generator of this action.

We refer to the books  \cite{kn:JKO} and  \cite{kn:KoLaOl12} for further details of the dynamical system $x.\, \omega$ and its generator.

\medskip
\noindent
{\bf Assumption 1}.   The action $(x,\omega)\mapsto x.\, \omega.$ preserves the measure $\mathbb Q$  and  is ergodic.

\bigskip
\noindent
Let $\sigma$ be a measurable symmetric $d\times d$ matrix valued function defined on   $\Omega$.

\medskip
\noindent
{\bf Assumption 2}.  There is  $\varkappa>0$ such that the following estimates hold:
$$
\varkappa |\zeta|^2\leq |\sigma(\omega)\zeta|^2\leq \varkappa^{-1} |\zeta|^2, \quad \hbox{for all }\omega\in\Omega\hbox{ and } \zeta\in\mathbb R^d.
$$

\bigskip
\noindent
 Let  $\mathcal{D}=\{g\in L^2(\Omega)\,;\, Dg\in (L^2(\Omega))^d\}$  be  the $L^2$ domain of the following bilinear form:
 $$
 (f,g) \longrightarrow \frac{1}{2}\int_\Omega\sigma Df\cdot\sigma Dg\,d\mathbb Q =: \mathfrak{E}(f,g).
 $$
The bilinear form $\mathfrak{E}(f,g)$ with domain  $\mathcal{D}$ is a Dirichlet form. We postulate the existence of a  Hunt process
with continuous paths whose Dirichlet form is $(\mathfrak{E}, \mathcal{D})$. We denote by $\omega(s)$ the coordinate process
on path space  $C(\mathbb R^+, \Omega)$. We denote by $\mathbb P_0$ the law of the Hunt process with initial law $\mathbb Q$.

 We also introduce the subspaces of centered functions
 $$
 L^2_0(\Omega)=\left\{u\in L^2(\Omega)\,:\, \int_\Omega u\,d\mathbb Q=0\right\},\qquad
 \mathcal{D}_0=\left\{u\in \mathcal{D}\,:\, \int_\Omega u\,d\mathbb Q=0\right\}.
 $$
Due to the ergodicity,  the quadratic form
$$
\mathfrak{E}(f)= \frac{1}{2}\int_\Omega\sigma Df\cdot\sigma Df\,d\mathbb Q
 $$
 defines a norm on $\mathcal{D}_0$.  We introduce $H^1(\Omega)$ as the completion of $\mathcal{D}_0$
with respect to $\mathfrak{E}$. By construction,     $H^1(\Omega)$  is a Hilbert space.

We then define $H^{-1}(\Omega)$ as the dual space to $H^1(\Omega)$.
Let $\mathcal A$ be the linear subset of $L_0^2(\Omega)$ consisting of functions  $f\in L_0^2(\Omega)$ such that
for some constant $c$ and for any $u\in \mathcal{D}_0$ the following inequality holds
\begin{equation}\label{h-1}
\Big(\int_\Omega fu\,d\mathbb Q\Big)^2\leq c^2 \mathfrak{E}(u).
\end{equation}
The map $u\to \int_\Omega fu d\mathbb Q$ defines an element in $H^{-1}(\Omega)$ whose norm is the smallest constant
$c$ for which inequality \eqref{h-1} holds true,  so that we can interpret  $\mathcal{A}$
as a subset of $H^{-1}(\Omega)$.   Then $\mathcal{A}$ is dense in $H^{-1}(\Omega)$.
With this construction we may identify $\mathcal{A}$ with $L^2_0(\Omega)\cap H^{-1}(\Omega)$. In what follows we use the latter notation.

 Let $L^2_{\rm pot}(\Omega)$ be the closure of $ \{v=Du\,:\,u\in \mathcal{D}_0\}$
in  the space $\big(L^2(\Omega)\big)^d$ equipped with the norm $(\frac{1}{2}\int_\Omega|\sigma v|^2\,d\mathbb Q)^{1/2}$.
By construction $L^2_{\rm pot}(\Omega)$ is a Hilbert space.

Let $f\in L^2_0(\Omega)\cap H^{-1}(\Omega)$. Setting
$$
\langle f,Du\rangle=\int_\Omega f u\,d\mathbb Q
$$
we can interpret $f$ as a linear continuous functional on $L^2_{\rm pot}(\Omega)$.  Using the Riesz theorem
we identify $f$ with an element $\tilde f\in L^2_{\rm pot}(\Omega)$. In other words, $\widetilde f$ is the unique element of
$L^2_{\rm pot}(\Omega)$ such that
$$
\int_\Omega fu \,d\mathbb Q=\frac{1}{2} \int_\Omega \sigma\widetilde f\cdot\sigma Du\,d\mathbb Q \quad\hbox{for all }u\in\mathcal{D}_0.
$$
 Observe that the map $f\mapsto\tilde f$ preserves the norms in $H^{-1}(\Omega)$ and   $L^2_{\rm pot}(\Omega)$.
 Therefore, it extends to an isometry between $H^{-1}(\Omega)$ and   $L^2_{\rm pot}(\Omega)$.

Let us introduce the notation
$$
\Sigma(f)=2\|f\|^2_{H^{-1}(\Omega)}=\int_\Omega |\sigma\tilde f|^2\,d\mathbb Q
$$
 and
 \begin{equation}\label{dsig}
\Sigma(f,g)=2( f,g)_{H^{-1}(\Omega)}=\int_\Omega \sigma\tilde f\cdot\sigma\tilde g \,d\mathbb Q.
\end{equation}

\subsection{Invariance principle}\label{ss_inp}
Given a square integrable and centered function $f\,:\,\Omega\mapsto \mathbb R$ satisfying   (\ref{h-1}),
and given a continuous trajectory $(\omega(s)\,;\,s\geq 0)$ in $\Omega$, we set
$$
A_f(t)=\int_0^t f(\omega(s))\,ds.
$$
Observe that  the process $(A_f(t)\,;\,t\geq 0)$ is an additive functional of
the  process $(\omega(t)\,;\,t\geq 0)$.    As was proved in  \cite{kn:KV}, the following invariance principle
holds:
\begin{theo}\label{t_KiVa}
Let $f\,:\,\Omega\mapsto \mathbb R$ be a square integrable and centered function satisfying   (\ref{h-1}).
Then under $\mathbb P_0$ the family of processes $(A^\eps_f(t)=\eps A_f(t/\eps^2)\,;\,t\geq 0)$  converges in law,
as $\eps\to 0$, in $C([0,\infty),\mathbb R)$ towards a Brownian motion with variance $\Sigma(f)$.
Moreover,
$$
\frac{1}{t}\mathbb E_0\left[A_f^2(t)\right]\  \mathop{\longrightarrow}\limits_{t\to\infty}\ \Sigma(f).
$$
\end{theo}

In fact, the approach of  \cite{kn:KV} provides a martingale approximation for $A_f$.    It then follows
that for any finite collection  $(f_1,\ldots, f_n)$ of functions satisfying the assumptions of the above theorem
the joint invariance principle holds for the $n$-dimensional additive functional
$(A_{f_1},\ldots, A_{f_n})$ with limit covariance matrix $\{\Sigma(f_i,f_j)\}_{i,j=1}^n$.
Moreover, if $(M_1,\ldots, M_k)$ are continuous square integrable  martingale additive functionals, then
the $(n+k)$-dimensional additive functional $(A_{f_1},\ldots, A_{f_n},M_1,\ldots, M_k )$ satisfies
the joint invariance principle.

\subsection{Extension to $H^{-1}$}\label{s_exten}
In this section we extend the previous result to all elements of $H^{-1}(\Omega)$. This extension relies on the following lemma.
\begin{lm}\label{l_doob}
For any  $g\,:\,\Omega\mapsto \mathbb R$ being a square integrable and centered function satisfying   (\ref{h-1})
and any $t>0$ we have
\begin{equation}\label{doob}
\mathbb E_0\left[\big(\sup\limits_{s\leq t}| A_g(s)|\big)^2\right]\leq 8t\|g\|^2_{H^{-1}(\Omega)}.
\end{equation}
\end{lm}

\begin{proof}
The proof relies on the forward-backward martingale representation  of $A_g$; see \cite[chapter 5.7]{kn:FOT}.
Denote by $r_t$ the time reversal operator at time $t$: \  $\omega\circ r_t(s)=\omega(t-s)$ for all $s\in [0,t]$.
Then,
\begin{equation}\label{ma_fb}
A_g(s)= \frac{1}{2}\big(M(s)+(M(t)-M(t-s))\circ r_t\big) ,
\end{equation}
where, under $\mathbb P_0$,   $M$ is a continuous square integrable martingale with
bracket
$$
\langle M\rangle(t)=\int_0^t|\sigma \tilde g|^2(\omega(s))\,ds.
$$
The first martingale on the right hand side of (\ref{ma_fb}) can be estimated using Doob's inequality as follows:
$$
\mathbb E_0 \big[\sup\limits_{s\leq t}|M(s)|^2\big]\leq
4\mathbb E_0  \big[\langle M\rangle(t)\big]=4t\int_\Omega |\sigma \tilde g|^2\,d\mathbb Q=8t\|g\|^2_{H^{-1}(\Omega)}.
$$
The second term can be treated in a similar way taking advantage of the fact that $\mathbb P_0$ is invariant with respect to $r_t$.
\end{proof}

The first consequence of the lemma is that we can make sense of $A_f$ for $f\in H^{-1}(\Omega)$.
 Observe that  $f\in L_0^2(\Omega)\cap H^{-1}(\Omega)$ vanishes as an element of $H^{-1}(\Omega)$ iff $f=0$ $\mathbb Q$-a.s.
Due to the lemma, the map $f\mapsto (A_f(t)\,;\,t\geq0)$  is linear continuous  from $L_0^2(\Omega)\cap H^{-1}(\Omega)$ equipped with
$H^{-1}(\Omega)$ topology to
$L^2(\Omega, C[0,\infty))$.
Since  $L_0^2(\Omega)\cap
H^{-1}(\Omega)$ is dense in $H^{-1}(\Omega)$, this map extends to a linear continuous map on $H^{-1}(\Omega)$.
We will sometimes abuse notation and keep the notation
$$A_f(t)=\int_0^t f(\omega(s))\,ds$$
for $f\in H^{-1}(\Omega)$.

The following extension of Theorem \ref{t_KiVa} follows from Lemma \ref{l_doob}.
\begin{theo}\label{t_KiVa_h-1}
Let $f\in H^{-1}(\Omega)$.
Then under $\mathbb P_0$ the family of processes \\ $(A^\eps_f(t)=\eps A_f(t/\eps^2)\,;\,t\geq 0)$  converges in law,
as $\eps\to 0$, in $C[0,\infty)$ towards a Brownian motion with variance $\Sigma(f)$.
Moreover,
$$
\frac{1}{t}\mathbb E_0\left[A_f^2(t)\right]\  \mathop{\longrightarrow}\limits_{t\to\infty}\ \Sigma(f).
$$
\end{theo}

Notice that as in Theorem \ref{t_KiVa},  for any finite collection  $(f_1,\ldots, f_n)$ of elements of $H^{-1}(\Omega)$,
the joint invariance principle holds for the vector
$(A_{f_1},\ldots, A_{f_n})$ with  limit covariance matrix $\{\Sigma(f_i,f_j)\}_{i,j=1}^n$.  If $(M_1,\ldots, M_k)$ are continuous square integrable martingale additive functionals, then
the $(n+k)$-dimensional additive functional $(A_{f_1},\ldots, A_{f_n},M_1,\ldots, M_k )$ satisfies
the joint invariance principle.

For $f,\,g\in H^{-1}(\O)$ we have
\begin{equation}\label{cro_co}
\frac{1}{t}\mathbb E_0\left[A_f(t)A_g(t)\right]\  \mathop{\longrightarrow}\limits_{t\to\infty}\ \Sigma(f,g).
\end{equation}

\subsection{Lebowitz-Rost type results}\label{s_leb-rost}

Let $(M(t)\,;\,t\geq 0$ be a continuous martingale additive functional of the Markov process $\omega(\cdot)$. Then $M(t)$
is a continuous martingale with stationary increments under $\mathbb P_0$.  We assume that its bracket
is of the form
$$
\langle M\rangle(t)= \int_0^t m(\omega(s))\,ds
$$
with $m\in L^\infty(\Omega)$.

For  $\lambda\in \mathbb R$, let  $\mathbb P_0^\lambda$ be the measure on path space that satisfies
$$
\frac{d\mathbb P_0^\lambda\big|_{\mathcal F_t}}{d\mathbb P_0\big|_{\mathcal F_t}}=e^{\lambda M(t)-
\frac{\lambda^2}{2}\langle M\rangle(t)}
$$
for all $t\geq0$.

It follows from our assumptions that $M(\cdot)$ satisfies the invariance principle. Let $f\in H^{-1}(\Omega)$. Observe that
the pair $(A_f,M)$ satisfies the joint invariance principle under $\mathbb P_0$. We denote by $\Gamma_M$ the off-diagonal term of the limit covariance
matrix. It follows from the assumptions on $M(\cdot)$ and
(\ref{doob}) that
$$
\Gamma_M(f)=\lim\limits_{t\to\infty}\frac{1}{t}\mathbb E_0\left[A_f(t)M(t)\right].
$$

\begin{theo}\label{t_LR}
Let $f\in H^{-1}(\Omega)$, and let $\alpha$ be a positive real number.
Then under $\mathbb P_0^\lambda$ the family of processes $(A^\eps_f(t)=\eps A_f(t/\eps^2)\,;\,t\geq 0)$  converges in law in $C[0,\infty)$,
as $\eps\to 0$, $\lambda\to 0$ and $\lambda^2/\eps^2\to\alpha$,  towards a Brownian motion with  variance $\Sigma(f)$ and  constant drift $\sqrt{\alpha}\,\Gamma_M(f)$.
\end{theo}
The statement of this theorem remains valid in the multi dimensional case. Namely, let $f_1,\ldots, f_n$ belong to $H^{-1}(\Omega)$,
and let $M_1,\ldots, M_k$ be continuous square integrable martingale additive functionals. Let $M_j^\eps(t)=\eps M_j(t/\eps^2)$, $j=1,\ldots,k$.
Then, as $\eps\to 0$, $\lambda\to 0$ and $\lambda^2/\eps^2\to\alpha$, under $\mathbb P_0^\lambda$ the rescaled family $(A^\eps_{f_1},\dots, A^\eps_{f_n},M^\eps_1,\ldots, M^\eps_k)$ converges in law in $C([0,\infty),\mathbb R^{n+k})$ to a Brownian motion with constant drift.  The limit covariance of  $A_{f_i}$ and $A_{f_j}$  is
$\Sigma(f_i,f_j)$; the limit covariance of   $A_{f_i}$ and $M_j$ is $\Gamma_{M_j}(f_i)$.
\begin{proof}
The arguments below are essentially borrowed from \cite{kn:LR}. Let $F$ be a continuous bounded functional on path space on time interval $[0,T]$, $T>0$.
Then we have
\begin{equation}\label{flr}
\mathbb E_0^\lambda\big[F(A_f^\eps(t)\,;\,t\in[0,T])\big]
=\mathbb E_0\big[F(A_f^\eps(t)\,;\,t\in[0,T]) e^{\lambda M(T/\eps^2)-(\lambda^2/2)\langle M\rangle(T/\eps^2)}\big].
\end{equation}
By Theorem \ref{t_KiVa_h-1} and since $\lambda^2/\eps^2$ tends to $\alpha$, under $\mathbb P_0$ the law of $(A_f^\eps\,,\,
\lambda M(\cdot/\eps^2))$ converges to the law of a two-dimensional Brownian motion $\check Z=(\check Z_1,\check Z_2)$ defined on a probability space
 $(\mathcal{W}, \mathcal{F}, \mathcal{P})$. Let $\mathcal E$ denote integration with respect to $\mathcal P$.

Let $\Sigma_2=\{(\Sigma_2)_{ij}\}_{i,j=1}^2$ be the covariance matrix of $\check Z$. It follows from the definitions that $(\Sigma_2)_{11}=\Sigma(f)$ , and
 $(\Sigma_2)_{12}=(\Sigma_2)_{21}=\sqrt{\alpha}\Gamma_M(f)$.
 Notice also that
 $\mathcal{E}[(\check Z_2(T))^2]
 =\alpha\mathbb E_0[\langle M\rangle(T)]=\alpha\mathbb E_0[\langle M\rangle(1)]T$.
 By the ergodic theorem, the process $\lambda^2
 \langle M\rangle(\cdot/\eps^2)$ converges in probability under $\mathbb P_0$ to the deterministic process $(\alpha^2\mathbb E [\langle M\rangle(1)]t\,;\,t\geq0)$.\\
 Therefore, the triple $(A_f^\eps\,,\,\lambda M(\cdot/\eps^2)\,,\, \langle M\rangle(\cdot/\eps^2) )$
 converges in law under $\mathbb P_0$ towards the process $\big((\check Z(t), \alpha^2\mathbb E [\langle M\rangle(1)]t)\,;\,t\geq0\big)$.

Besides, under the assumption that $m\in L^\infty(\Omega)$, we can estimate
$$
\mathbb E_0\left[e^{2\lambda M(T/\eps^2)}\right]=\mathbb E_0\left[e^{2\lambda M(T/\eps^2)-2\lambda^2\langle M\rangle(T/\eps^2)}
e^{2\lambda^2\langle M\rangle(T/\eps^2)}\right]$$
$$\leq e^{2\alpha\|m\|_{L^\infty(\Omega)}}.
$$
Therefore, we can pass to the limit in  (\ref{flr}), and  the right-hand side converges to\\
$\mathcal{E}\big[F(\check Z_1(t)\,,\,t\in[0,T])e^{\check
Z_2(T)-\mathcal{E}[(\check Z_2(T))^2]}\big] $.
The Gaussian integration by parts formula yields
$$
\mathcal{E}\big[F(\check Z_1(t)\,,\,t\in[0,T])e^{\check
Z_2(T)-\mathcal{E}[(\check Z_2(T))^2]}\big] = \mathcal{E}\big[F(\check Z_1(t)+\sqrt{\alpha}\Gamma_M(f)t\,,\,t\in[0,T])\big].
$$

The extension to the multidimensional case described in the comment that follows the Theorem  is an immediate consequence of the joint
invariance principle stated just after Theorem \ref{t_KiVa_h-1}.
\end{proof}

\subsection{Diffusions in a random environment}
\label{ss_appl}

In this section we apply the above results to the case of a diffusion  in random environment.
We choose for $\Omega$ the space of smooth $d\times d$ symmetric matrix functions defined on $\mathbb R^d$.
We equip this space with the topology of uniform convergence on compact subsets of $\mathbb R^d$. Besides,
$\mathbb R^d$ acts on $\Omega$ by additive translations.

Let $\mathbb Q$ be a stationary ergodic measure on $\Omega$ so {\bf Assumption 1} holds. Choose $\sigma$ satisfying {\bf Assumption 2}.
We define $\sigma^\omega(x)=\sigma(x.\omega)$ for $x\in\mathbb R^d$.   We further assume

\medskip\noindent
{\bf Assumption 3:} for any environment $\o$, the function
$x\rightarrow \sigma^\o(x)$ is smooth.

\medskip\noindent
We  introduce the notation
$$a^\o=(\sigma^\o)^2 \hbox{\, and\, } b^\o=\frac 12 \mathrm{div}a^\o\,.$$
Observe that both $a^\o$ and $b^\o$ are then stationary fields i.e.
$a^\o(x)=a(x.\o)$ and $b^\o(x)=b(x.\o)$ for some functions $a=\sigma^2$ and $b$.
It is immediate to check that $b$ belongs to $(H^{-1}(\Omega))^d$.

Let $(W_t:t\geq 0)$ be a Brownian motion defined on some probability
space $(\WW,\FF,P)$. We denote expectation with respect to $P$ by
$E$.
We define the process $X_x^\o$ as the solution of the following stochastic differential equation
\beqn \label{eq:sde1}
dX_x^\o(t)=b^\o(X_x^\o(t))\,dt+\sigma^\o(X_x^\o(t))\,dW_t \,;\,
X_x^\o(0)=x\,. \eeqn
Then $X^\o$ is a Markov process
generated by the operator
\beqn \label{eq:gen} \LL^\o f(x)=\frac 1 2\, \mathrm{div}(\,a^\o\,\nabla f)(x)\,.
\eeqn
Define the process $\o(t)=X_0^\o(t).\o$.
\begin{prop}\label{p_dom}
Under $P$,  the process $\omega(\cdot)$ is a symmetric Hunt process with reversible measure
 $\mathbb Q$ and Dirichlet form $(\mathfrak{E},\mathcal{D})$ in $L^2(\Omega, \mathbb Q)$.
\end{prop}

\begin{proof}
It is clear that $\omega(\cdot)$ is a Hunt process with continuous paths. Since the generator $\mathcal{L}^\omega$  is symmetric, the Lebesgue measure
is reversible for the process $X^\omega_x$ for all $\omega$. This combined with the fact that $\mathbb Q$ is stationary implies that
the measure $\mathbb Q$ is reversible for the process  $\omega(\cdot)$.\\
Now we identify the Dirichlet form of $\omega(\cdot)$.
For a given $\omega$ the domain of the Dirichlet form of the process $X^\omega_x$ is $H^1(\mathbb R^d)$.
Let $F\in\mathcal{D}$. For $\omega\in\Omega$ we define $F^\omega(x)=F(x.\omega)$. Then for almost all $\omega$ the function
$F^\omega(\cdot)$ belongs to $H^1_{\rm loc}(\mathbb R^d)$ (see \cite[page 232]{kn:JKO}).
 From these two facts the desired statement follows.
\end{proof}

\bigskip
According to Proposition \ref{p_dom} we are in the framework of this Section.
Therefore, we set $\mathbb P_0(A)=\int_\Omega d\mathbb Q(\omega)P(X_0^\omega(\cdot).\omega\in A)$ for all measurable sets $A$ in the path space.

\begin{rmk}{\rm
One can retrieve the trajectory of $X^\omega_0$ from the trajectory $\omega(\cdot)$ looking for $x\in\mathbb R^d$ that solves the equation
\begin{equation}\label{enlarging}
x.\omega=\omega(t).
\end{equation}
If this equation has a unique solution $x$, then $X_0^\omega(t)=x$, and it follows from the structure of  equation  (\ref{enlarging})
that  $X^{\omega}_0$ is an additive functional of the process $\omega(\cdot)$.   Furthermore,
enlarging the space $\Omega$ if necessary, we may always assume that  equation (\ref{enlarging}) has a unique solution.
For instance, let $(V_1,\,\ldots, V_d)$ be independent nonconstant random fields with finite range of correlation indexed
by $\mathbb R$ and defined on some probability space $\Omega'=\Omega_1\times\ldots\times\Omega_d$. We assume that each $\Omega_j$
is equipped with a measure preserving ergodic action of $\mathbb R$. For $\omega'=(\omega_1,\ldots,\omega_d)$ and $x\in\mathbb R^d$ we define
$x.\omega'= (x_1.\omega_1,\ldots,x_d.\omega_d)$, and let $V^{\omega'}(x)=V(x.\omega')$.   We enlarge $\Omega$ by taking the product space $\Omega\times\Omega'$.
Observe that if the equation  (\ref{enlarging}) has two different solutions then one of the components of $V^\omega$ is periodic, and this happens with probability $0$.

A similar argument is used in  \cite[Remark 4.2]{kn:DFGW}.
}
\end{rmk}

The martingale part of $X^\omega_0$, that can be expressed  as $\int_0^\cdot \sigma(\omega(s))\,dW_s$, is
a martingale additive functional of the process $\omega(\cdot)$. The drift part is also an additive functional of the form $\int_0^\cdot b(\omega(s))\,ds$ with
$b\in (H^{-1}(\Omega))^d$.   Therefore, Theorem \ref{t_KiVa_h-1} and the comment following this theorem imply the joint invariance principle for these
processes. As a consequence, the family of processes $\big(\eps X_0^\omega(t/\eps^2)\,;\,t\geq 0\big)$ converges in law under $P\times\mathbb Q$, as $\eps\to0$, towards a Brownian
motion with the {\em effective covariance} that we denote by $\Sigma$, and
$$
e\cdot\Sigma e=\lim\limits_{t\to\infty}\ \frac{1}{t}\int_\Omega E[(X_0^\omega(t)\cdot e)^2]\,d\mathbb Q(\omega),\qquad \hbox{for any }e\in\mathbb R^d.
$$

In the sequel we often use the notion of  symmetric and antisymmetric additive functionals of $\omega(\cdot)$.
For $T>0$ the {\em time reversal operator} $R_T$ maps a trajectory $(\omega(t)\,;\,0\leq t\leq T)$ to the trajectory
$(\omega(T-t)\,;\,0\leq t\leq T)$.
An additive functional  is called {\em symmetric}  with respect to time reversal if its restriction to the time interval $[0,T]$ is invariant
under $R_T$ for all $T$.  It is called {\em antisymmetric} if it changes sign upon the action of $R_T$.
For instance, $A_f$ is a symmetric additive functional whereas $X_0^\omega$ is antisymmetric.

Let $e_1$ be a non-zero vector and $\lambda>0$. We define $\hat\lambda$
to be the vector ${\hat\lambda}=\lambda e_1$.
We consider the perturbed stochastic differential equation:
\beqn \label{eq:sde2}
dX_x^{\l,\,\o}(t)=b^\o(X_x^{\l,\,\o}(t))\,dt+
\,a^\o(X_x^{\l,\,\o}(t))\hl\,dt+\sigma^\o(X_x^{\l,\,\o}(t))\,dW_t \,;\, X_x^{\l,\,\o}(0)=x\,.
\eeqn
Then $X^{\lambda,\omega}$ is a Markov process with generator
$$
\mathcal{L}^{\lambda,\omega}f(x)=\mathcal{L}^\omega f(x)+a^\omega(x)\hat\lambda\cdot\nabla f(x).
$$
Applying the Girsanov formula (see \cite{kn:RYor} ) to the processes $X^\o$ and
$X^{\l,\,\o}$, we get that, for any $\o$, \beqn\label{eq:gir}
E[F(X_0^{\l,\,\o}([0,t]))]=E[F(X_0^{\o}([0,t]))\, e^{\l\bB (t)
-\frac {\l^2}2 \langle\bB\rangle(t)}]\,, \eeqn where $\bB$ is the
martingale
\begin{equation}\label{mart_part}
 \bB (t)=\int_0^t \sigma^\o(X_0^\o(s))\, e_1 \cdot dW_s
\end{equation}
and $\langle\bB\rangle$ is its bracket
$$\langle\bB\rangle(t)=\int_0^t \vert \sigma^\o(X_0^\o(s))\, e_1 \vert^2 \, ds \,.$$
Observe that the process $\bar B$ is an additive functional of $\omega(\cdot)$ which can be written as
\begin{equation}\label{b-bar}
 \begin{array}{c}\displaystyle
\bB(t)-\bB(s)=\tr{e_1}\, (X_0^\o(t)-X_0^\o(s))-\int_s^t \tr{e_1}\, b^\o(X_0^\o(u))\,du\\[2mm] \displaystyle
= \tr{e_1}\, (X_0^\o(t)-X_0^\o(s))-\int_s^t \tr{e_1}\, b(\o(u))\,du\,.
\end{array}
\end{equation}
We let $\omega^\lambda(t)=X_0^{\lambda,\o}(t).\o$.
Then
the law of the process $\omega^\lambda(\cdot)$ with the initial measure   $\mathbb Q$ coincides with the measure $\mathbb P^\lambda_0$
defined in Section \ref{s_leb-rost}, where we set $M=\overline B$.
 Let $\bar\Gamma$ be the covariance operator defined in Section \ref{s_leb-rost} with $M=\overline B$.

\begin{prop}\label{p_cov}
Let $f\in H^{-1}(\Omega)$, Then, under $P\times\mathbb Q$, the processes
$$
\Big(\lambda X_0^{\lambda,\omega}(\lambda^{-2}\cdot),\, \lambda \int_0^{\lambda^{-2}\cdot} f(\omega^\lambda(s))\,ds\Big)
$$
converge in law in $C([0,\infty), \mathbb R^{d+1})$, as $\lambda$ tends to $0$,
towards a Brownian motion with constant drift.  The limit covariance matrix and the drift are given, respectively, by
$$
\left(
\begin{array}{ccc}
\Sigma &0\\[2mm]
0&\Sigma(f)
\end{array}
\right)\,,\qquad
\left(
\begin{array}{c}
\Sigma e_1\\[2mm]
\overline{\Gamma}(f)
\end{array}
\right)\,.
$$
\end{prop}

\begin{proof}
Theorem \ref{t_LR} and the comment following this theorem apply and yield the convergence in law of $\big(\lambda{X}_0^{\lambda,\omega}(\lambda^{-2}\cdot), \lambda{A}^{\lambda,\omega}_{0,f}(\lambda^{-2}\cdot)\big)$, under the annealed measure $P\times\mathbb Q$.

According to Theorem \ref{t_LR} the limit covariance matrix is also the limit covariance matrix under the annealed measure  of
\begin{equation}\label{no_lamb}
\Big(\frac{1}{\sqrt{t}}X_0^\omega(t),\,  \frac{1}{\sqrt{t}} \int_0^t f(\omega(s)\,ds\Big),
\end{equation}
as $t\to\infty$.  By definition of $\Sigma$, the covariance of the $X_0^\omega$ component converges to $\Sigma$, while the limit variance of the last
component is $\Sigma(f)$.  The covariance of $X_0^\omega(t)$  and  $\int_0^t f(\omega(s)\,ds$ vanishes because $X_0^\omega(\cdot)$ is an antisymmetric with respect to time reversal additive functional of $\omega(\cdot)$, and   $\int_0^t f(\omega(s)\,ds$ is symmetric.

As for the limit drift part, Theorem \ref{t_LR} implies that  it is given by the limit of the covariances of the vector in (\ref{no_lamb}) and $t^{-1/2}\overline{B}(t)$.  The contribution of the last
component is $\overline{\Gamma}(f)$ by definition.   To identify the contribution of the $X_0^\omega$ component we rely on formula (\ref{b-bar}) observing once again that the covariance of $X_0^\omega(t)$ and $\int_0^t e_1\cdot b(\omega(s))\,ds$ vanishes for symmetry reasons.
\end{proof}

\subsubsection{The corrector}
\label{sss_corr}

We recall that $b\in (H^{-1}(\Omega))^d$. Let $\widetilde b$ be the matrix whose columns are elements of $L^2_{\rm pot}(\Omega)$ such that
$\widetilde b e=\widetilde{b\cdot e}$ for any $e\in\mathbb R^d$. Let $\widetilde b^\omega(x)=\widetilde b(x.\omega)$ be the space realization of $\widetilde b$.  For any $e\in\mathbb R^d$ for almost all  $\omega\in\Omega$ then $\widetilde b^\omega\cdot e$ is a curl-free  function in $L^2_{\rm loc}(\mathbb R^d)$. Therefore, there exists a smooth vector valued function $\chi(\cdot,\omega)$ defined on $\mathbb R^d$ and such that  $\nabla (\chi(\cdot,\omega)\cdot e)=\widetilde b^\omega\cdot e$. The function $\chi$ is called a {\it corrector}. Observe that it is uniquely defined up to an additive constant.   By the definition of $\widetilde b$,
$$
\int_\Omega (b\cdot e)u\,d\mathbb Q=\frac{1}{2}\int_\Omega\sigma(\widetilde b\cdot e)\cdot \sigma\nabla u\,d\mathbb Q, \quad\hbox{for any }
u\in \mathcal{D}_0.
$$
Going to the space of realizations yields
$$
\int_{\mathbb R^d} b^\omega(x)\!\cdot \!e\, u(x)\,dx=\frac{1}{2}\int_{\mathbb R^d}\sigma^\omega(x)\nabla(\chi(x,\omega)\cdot e)\cdot
\sigma^\omega(x)\nabla u(x)\,dx
$$
for any $u\in C_0^\infty(\mathbb R^d)$. Integrating by parts we obtain
$$
\int_{\mathbb R^d} b^\omega(x)\!\cdot \!e\, u(x)\,dx=-\int_{\mathbb R^d}\mathcal{L}^\omega\chi(x,\omega)\!\cdot\! e\, u(x)\,dx.
 $$
Thus, $\mathcal{L}^\omega\chi\cdot e=-b^\omega\cdot e$.  This implies that for almost all $\omega\in\Omega$ the process
$\chi(X_x^\omega(t),\omega)+X_x^\omega(t)$ is a martingale under $P$ for all starting points $x$.

The following Proposition illustrates the role of the corrector. However, it will not be used in the sequel.

\begin{prop}\label{p_gamma}
Let $f\in H^{-1}(\Omega)\cap L^2(\Omega)$. Then
$$
\mathbb E_0 \Big[\frac{1}{t}\int_0^t f(X_0^\omega(s).\omega)\,\chi(X_0^\omega(s),\omega)\cdot e_1ds\Big]
\longrightarrow -\frac{1}{2}\bar\Gamma(f),\qquad\hbox{as }t\to\infty.
$$

\end{prop}

\begin{proof}
The proposition relies on the following statement. Recall that
 $\Sigma(f,g)$ is defined in (\ref{dsig}), see also \eqref{cro_co}.
\begin{lm}\label{l_cova}
 We have
 $$
\bar \Gamma(f)= -\Sigma(f, b\cdot e_1)\qquad \hbox{for all }f\in H^{-1}(\Omega).
 $$
\end{lm}

\begin{proof}[Proof of Lemma \ref{l_cova}]
By definition,
$$
\bar\Gamma(f)=\lim\limits_{t\to\infty}\frac{1}{t}\mathbb E_0\left[A_f(t)\overline B(t)\right]
$$
with $\overline B$ defined in \eqref{mart_part}.
Notice that
$$
\overline B(t)=(X_0^\omega(t)-X_0^\omega(0))\cdot e_1-A_{b\cdot e_1}(t).
$$
 As we already observed, $A_f$ is a symmetric additive functional and $(X_0^\omega(t)-X_0^\omega(0))$ is antisymmetric. Therefore, the covariance of $X_0^\omega(t)-X_0^\omega(0)$ and $A_f$ vanishes. Thus,
$$
\bar\Gamma(f)=-\lim\limits_{t\to\infty}\frac{1}{t}\mathbb E_0\left[A_f(t)A_{b\cdot e_1}(t)\right]= -\Sigma(f,b\cdot e_1).
$$
\end{proof}

Define
$$
m_t=\chi(X_0^\omega(t),\omega)\cdot e_1- \chi(0,\omega)\cdot e_1+ A_b(t)\cdot e_1.
$$
Then the process $\{m_t\,:\,t\geq 0\}$ is a martingale under $P$.  We have
\begin{equation}\label{aux_co10}
\begin{array}{c}
\displaystyle
\mathbb E_0\Big[\frac{1}{t}\int_0^t f(\omega(s))\chi(X_0^\omega(s),\omega)\cdot e_1\,ds
\Big] =\mathbb E_0\Big[\frac{1}{t}\int_0^t f(\omega(s))m_s\,ds
\Big]\\[4mm]
\displaystyle
 -\mathbb E_0\Big[\frac{1}{t}\int_0^t f(\omega(s))A_b(s)\cdot e_1\,ds
\Big] +\mathbb E_0\Big[\frac{1}{t}\int_0^t f(\omega(s))\chi(0,\omega)\cdot e_1\,ds
\Big].
\end{array}
\end{equation}
Using the martingale property of $m_\cdot$, we get
$$
\mathbb E_0\Big[\frac{1}{t}\int_0^t f(\omega(s))m_s\,ds\Big] =\mathbb E_0\Big[\frac{1}{t}\Big(\int_0^t f(\omega(s))\,ds\Big) m_t
\Big]
$$
$$
=\mathbb E_0\Big[\frac{1}{t} A_f(t)\big(\chi(X_0^\omega(t),\omega)-\chi(0,\omega)+A_b(t)\big)\cdot e_1\Big]
=\mathbb E_0\Big[\frac{1}{t} A_f(t)A_b(t)\cdot e_1\Big];
$$
here we have also used the fact that $A_f$ is a symmetric with respect to time reversal additive functional  and  $\chi(X_0^\omega(\cdot),\omega)-\chi(0,\omega)$ is antisymmetric. Therefore, their covariance vanishes.

By stationarity and reversibility we have  for all $s\leq v$
$$
\mathbb E_0\Big[f(\omega(s))b(\omega(v))\Big]=\mathbb E_0\Big[f(\omega(v))b(\omega(s))\Big].
$$
Therefore,
$$
\mathbb E_0\Big[\frac{1}{t}\int_0^t f(\omega(s))A_b(s)\cdot e_1\,ds\Big]
=\mathbb E_0\Big[\frac{1}{t}\int_0^t f(\omega(s))\int_0^s b(\omega(v))\cdot e_1\,dvds\Big]
$$
$$
=\mathbb E_0\Big[\frac{1}{t}\int_0^t f(\omega(s))\int_s^t b(\omega(v))\cdot e_1\,dvds\Big]
=\frac{1}{2}\mathbb E_0\Big[\frac{1}{t} A_f(t)A_b(t)\cdot e_1\Big],$$
and we conclude that
\begin{equation}\label{aux_co100}\mathbb E_0\Big[\frac{1}{t}\int_0^t f(\omega(s))\chi(X_0^\omega(s),\omega)\cdot e_1\,ds\Big] =
\frac 12 \mathbb E_0\Big[\frac{1}{t} A_f(t)A_b(t)\cdot e_1\Big]+\mathbb E_0\Big[\frac{1}{t}\int_0^t f(\omega(s))\chi(0,\omega)\cdot e_1\,ds
\Big].
\end{equation}

As $t\to\infty$,  according to Theorem \ref{t_KiVa_h-1}, the term $\mathbb E_0\Big[\frac{1}{t} A_f(t)A_b(t)\cdot e_1\Big]$
converges to $\Sigma(f,b\cdot e_1)$. By the Ergodic theorem the last term on the right-hand side of (\ref{aux_co100}) converges to zero. Thus,
$$
\mathbb E_0 \Big[\frac{1}{t}\int_0^t f(X_0^\omega(s).\omega)\,\chi(X_0^\omega(s),\omega)\cdot e_1ds\Big]
\longrightarrow\frac{1}{2}\Sigma(f,b\cdot e_1)=-\frac{1}{2}\bar\Gamma(f),\qquad\hbox{as }t\to\infty.
$$

\end{proof}

\begin{rmk}\label{r_corr}
For a function $g\in L^2(\Omega)$ by stationarity we have
$$
\int_\Omega fg\,d\mathbb Q=\frac{1}{t}\mathbb E_0\Big[\int_0^tf(\omega(s))g(\omega(s))\,ds\Big].
$$
In general, $\chi(x,\omega)$ is not of the form $g(x.\omega)$.   This suggests that the expression
$$
\mathbb E_0 \Big[\frac{1}{t}\int_0^t f(X_0^\omega(s).\omega)\,\chi(X_0^\omega(s),\omega)\cdot e_1ds\Big]
$$
need not have a limit for all $f\in L^2(\Omega)$. However, the Proposition says that the limit exists for all
$f\in H^{-1}(\Omega)$. In this respect, $-\frac{1}{2}\bar\Gamma(f)$ can be interpreted as a substitute for the integral of
a function $f$ against the corrector $\chi$.

In the case of finite range of dependence and $d\geq 3$, then the corrector exists and
$$
-\frac12\bar\Gamma(f)=\int_\O f\chi d\mathbb Q,
$$
see \cite{kn:GO} and \cite{kn:GMa}.

\end{rmk}


\section{Continuity of steady states}\label{s_sec1}

In this section, we study continuity properties of the steady state $\nu_\lambda$ as $\l$ tends to $0$.
In particular we shall prove Theorem \ref{theo:continuityIntro}.
Our main tool is Lemma \ref{lm:estimateAlambda}. It will also be useful in the other sections of the paper.
An alternative version of Lemma \ref{lm:estimateAlambda}, which also implies Theorem \ref{theo:continuityIntro},  is given in Appendix B.

\subsection{The spaces $H^{-1}_\infty(\Omega)$ and $\tilde{H}^{-1}_\infty(\Omega)$.}

Let $F$ be a vector-valued function in $\big(L^\infty(\Omega)\big)^d$.

The formula
$$
\langle F,u\rangle=-\int_\Omega F\cdot Du\,d\mathbb Q
$$
defines a linear continuous functional on $H^1(\Omega)$. Therefore there exists an element $f\in H^{-1}(\Omega)$
such that
$\langle F,u\rangle$ is the duality product $\langle f,u\rangle_{H^{-1},H^1}$. We denote $f$ by $\div F$ as it coincides with the
standard divergence if $F$ is regular enough. Indeed, if  $F=(F_1,\ldots, F_d)$ is such that $F_j$ belongs to the domain of $D_j$,
then $\langle f,u\rangle_{H^{-1},H^1}=-(F,Du)_{L^2(\Omega)}=(\sum D_jF_j,u)_{L^2(\Omega)}=(\mathrm{div}F,u)_{L^2(\Omega)}$
for any $u\in H^1(\Omega)$. The second relation here follows from the fact that
$\sqrt{-1}D_j$ is a self-adjoint operator in $L^2(\Omega)$.  We define $H^{-1}_\infty(\Omega)$ to be the set of elements $f$ in
$H^{-1}(\Omega)$ of the form $f=\div F$ for some $F$ in $\big(L^\infty(\Omega)\big)^d$. Let
$$
\Vert f\Vert_{H^{-1}_\infty(\Omega)}=\min\{\Vert F\Vert_\infty\,;\, \div F=f\}.
$$
Then $H^{-1}_\infty(\Omega)$ is a Banach space. Indeed, it is clear that $\Vert f\Vert_{H^{-1}_\infty(\Omega)}$ is a norm. We have to check that $H^{-1}_\infty(\Omega)$ is complete
with respect to this norm. To this end consider a Cauchy sequence $\{f_m\}_{m=1}^\infty$  in  $H^{-1}_\infty(\Omega)$. Taking a subsequence $\{m_j\}$ we can assume that
$\|f_{m_{j+1}}-f_{m_{j}}\|\big._{H^{-1}_\infty}\leq 2^{-(j+1)}$. Then there exist
$\widetilde F_j\in \big(L^\infty(\Omega)\big)^d $ such that $\|\widetilde F_j\|\big._{\infty}\leq 2^{-j}$ and
 $f_{m_{j+1}}-f_{m_{j}}=\mathrm{div} \widetilde F_j$.  Denote $F=F_{m_1}+\sum_{j=1}^\infty\widetilde F_j$ with $\mathrm{div} F_{m_1}=f_{m_1}$ and $F_{m_1}\in \big(L^\infty(\Omega)\big)^d$.
By construction $F\in L^\infty(\Omega)$ and thus $f:=\mathrm{div} F\in H^{-1}_\infty(\Omega)$.
One can easily check that $f_{m_j}$ converges to $f$ in $H^{-1}_\infty(\Omega)$ as $j\to\infty$, and, by the triangle inequality, $f_m$ converges to $f$ as $m$ tends to $\infty$.

Observe that, for a given $f$ in $H^{-1}_\infty(\Omega)$ there may be several $F$'s in $\big(L^2(\Omega)\big)^d$ such that
$\div F=f$. They are characterized by the fact that $a^{-1}F+\frac 12 \widetilde{f}$ is orthogonal to $L^2_{pot}(\Omega)$.

\medskip
We call a function $f$ - or more generally an element $f$ in $H^{-1}(\Omega)$ - {\it local} if there exists
$R_f$ such that $f$ is measurable with respect to the $\sigma$-field $\HH_{B_{R_f}}$ where
$B_R$ is the ball of radius $R$.

\medskip
We denote by $\tilde{H}^{-1}_\infty(\Omega)$ the closure of the set of elements $f$ in $H^{-1}_\infty(\Omega)$ for which there exists
a bounded and local $F$ such that $\div F=f$.

\subsection{Proof of Theorem \ref{theo:continuityIntro}.}

In Section \ref{s_exten}, we defined the continuous additive functional $A_f$ for $f\in H^{-1}(\Omega)$.
Since, for all $t>0$, for all $\omega$, the laws of the processes $(X_0^\omega(s);0\leq s\leq t)$ and
$(X_0^{\lambda,\omega}(s);0\leq s\leq t)$ are equivalent, the same approximation procedure as in Section \ref{s_exten}
can be used to give a meaning to the continuous additive functional
$$A^{\lambda,\omega}_{0,f}(t):=\int_0^t f(\omega^\lambda(s))\,ds,$$
for $\mathbb Q$ almost all $\omega$.

We then have the following

\begin{lm}\label{lm:estimateAlambda}
For all $p\geq 1$ there exists a constant ${\tt C}_p$ such that for all $0<\lambda\leq 1$, for all  $f\in H^{-1}_\infty(\Omega)$,
for $\mathbb Q$ almost all  $\omega$,  for each  $t\geq 1/\lambda^2$  the following estimate holds
\begin{equation}\label{eq:upH-1}
E\Big(\max\limits_{0\leq s\leq t}|A^{\lambda,\omega}_{0,f}(s)|^p\Big)\leq {\tt C}_p\lambda^p t^p \Vert f\Vert_{H^{-1}_\infty(\Omega)}^p;
\end{equation}
the constant  ${\tt C}_p$ depends only on the ellipticity constant $\varkappa$ in {\bf Assumption 2} and the dimension.
\end{lm}

\begin{proof}

Let us first observe that after multiplying $f$ by an appropriate constant, we may assume that $\Vert f\Vert_{H^{-1}_\infty(\Omega)}< 1$.
We then choose $F$ in $\big(L^\infty(\Omega)\big)^d$ such that $f=\div F$ and $\sup_\omega \vert F(\omega)\vert< 1$.

\medskip

We then consider processes taking on values in $\mathbb R^{d+1}$. We use the notation $z=(x,y)$, $x\in\mathbb R^d$ and $y\in\mathbb R$.
Let us introduce the process
$$
Z_z^{\lambda,\omega}(t)=\big(X_x^{\lambda,\omega}(t),\, y+A_{x,f}^{\lambda,\omega}(t)+W^1_t\big),
$$
where $W^1$ is an independent one-dimensional Brownian motion (which is assumed to be defined on the same probability space
$(\mathcal{W},\mathcal{F},P)$ as $W$), and
$$
A_{x,f}^{\lambda,\omega}(t)=\int_0^t f(X_x^{\lambda,\omega}(s).\omega)\,ds.
$$
Notice that $Z^{\lambda,\omega}$ is a Markov process with generator
$$\mathcal{M}^{\lambda,\omega}=
(\mathcal{L}^{\lambda,\omega})_x+\frac{1}{2}\partial_y^2+f(x.\omega)\partial_y,
$$
where, for a function $q(z)$, the operator $(\mathcal{L}^{\lambda,\omega})_x$ acts on $q$ as a function of variable $x$.

Let us check that the operator $\mathcal{M}^{\lambda,\omega}$ can be written in the form
\beqn\label{eq:reversible}
\mathcal{M}^{\lambda,\omega}q =\frac 12 \div_x(a^\omega\nabla_xq)+\lambda a^\omega e_1\nabla_xq +
\div_x(F^\omega \partial_y q)-\partial_y(F^\omega\nabla_x q)+ \frac 12 \partial^2_{y}q,
\eeqn
where as above we use the notation $F^\omega(z)=F(x.\omega)$.
Indeed, since $F^\omega$ does not depend on $y$, we have
$$
\div_x(F^\omega \partial_y q)-\partial_y(F^\omega\nabla_x q)=
(\div_x F^\omega)\partial_y q=f^\omega\partial_y q.
$$
This implies the desired representation.  In the variables $\tilde z=\lambda z$ and $\tilde t=\lambda^2 t$, the
generator reads
\begin{equation}\label{gen_z}
\frac 12 \div_{\tilde x}(a^\omega(\lambda^{-1}\tilde x)\nabla_{\tilde x}q)+ a^\omega((\lambda^{-1}\tilde x)e_1\nabla_{\tilde x}q +
\div_{\tilde x}(F^\omega(\lambda^{-1}\tilde x) \partial_{\tilde y} q)-\partial_{\tilde y}(F^\omega(\lambda^{-1}{\tilde x})\nabla_{\tilde x} q)+\frac 12  \partial^2_{{\tilde y}}q.
\end{equation}

\medskip

Note that all the coefficients of the operator in (\ref{gen_z}) are bounded.
The parabolic Aronson estimates (see \cite[Theorems 8 and 9]{kn:aronson}) therefore hold uniformly in $\lambda$ and in $\omega$ on any finite time interval and in any fixed  ball.

Denote $\widetilde T_r=\inf\{s>0\,:\, |\lambda Z^{\lambda,\omega}_0(s/\lambda^2)|=r\}$.  Applying Aronson's lower bound  to the parabolic equation with  generator given by (\ref{gen_z}), we obtain that there exists $\delta_0>0$ such that for all $\lambda\in(0,1)$ and all $\omega$ such that $\sup\vert F^\omega\vert\leq 2$, then
$$
P(\widetilde T_1\geq 1)\geq \delta_0.
$$
Therefore,
$$
E\big(e^{-\widetilde T_1}\big)\leq 1-\eps_0
$$
for some $\eps_0>0$. Applying the Markov property we deduce that
$$
E\big(e^{-\widetilde T_r}\big)\leq (1-\eps_0)^{r-1},
$$
and
$$ P({\tilde T}_r\leq t)\leq e^t (1-\eps_0)^{r-1}.$$

Let $T\leq 1$. Since the events $({\tilde T}_r\leq T)$ and $(\lambda\max\limits_{0\leq s\leq \lambda^{-2}T}|Z^{\lambda,\omega}_0(s)|\geq r)$ coincide, we get that
\begin{equation}\label{addi1}
\begin{array}{c}
\displaystyle
E\Big(\lambda^p\max\limits_{0\leq s\leq \lambda^{-2}T}|Z^{\lambda,\omega}_0(s)|^p\Big)\\[4mm]
\displaystyle
=p\int_0^\infty r^{p-1}dr\, P({\tilde T}_r\leq T)\leq pe^T\int_0^\infty r^{p-1}dr\, (1-\eps_0)^{r-1}\leq \eta_0,
\end{array}
\end{equation}
where $\eta_0=pe\int_0^\infty r^{p-1}dr\, (1-\eps_0)^{r-1}$ is a constant.

Let now $T>1$ with integer part $[T]$. Note that
$$ \max\limits_{0\leq s\leq \lambda^{-2}T}|Z^{\lambda,\omega}_0(s)|
\leq \sum_{j=0}^{[T]}\max\limits_{j\lambda^{-2}\leq s\leq (j+1)\lambda^{-2}}  |Z^{\lambda,\omega}_0(s)-Z^{\lambda,\omega}_0(\lambda^{-2}j)|.$$
Therefore
$$E\Big(\lambda^p\max\limits_{0\leq s\leq \lambda^{-2}T}|Z^{\lambda,\omega}_0(s)|^p\Big)
\leq (T+1)^p \eta_0.$$
If we change variable to $t=\lambda^{-2}T$, we obtain
\begin{equation}\label{addi2}
E\Big(\max\limits_{0\leq s\leq t}|Z^{\lambda,\omega}_0(s)|^p\Big)\leq \lambda^{-p} (\lambda^2t+1)^p\eta_0.
\end{equation}
On the other hand
$$ E\Big(\max\limits_{0\leq s\leq t}|W^1_s|^p\Big)\leq C_p t^{p/2}.$$
Combining \eqref{addi1}, \eqref{addi2} and the last estimate and considering the lower bound $t\geq\lambda^{-2}$, we obtain the desired inequality.

\end{proof}

\begin{proof}[Proof of Theorem \ref{theo:continuityIntro}]
Apply Lemma \ref{lm:estimateAlambda}.
\end{proof}

\section{Construction of steady states}\label{s_sec2}

The goal of this section is to prove the existence of the steady state and weak steady state under the assumption
of finite range of dependence and get an explicit formula in terms of regeneration times, see formula (\ref{nu_f-def}).

In this section we assume {\bf Assumptions 1--4} to hold.

\medskip

We recall that a function $f$  or an element $f$ in $H^{-1}(\Omega)$ is local if there exists
$R_f$ such that $f$ is measurable with respect to the $\sigma$-field $\HH_{B_{R_f}}$ where
$B_R$ is the ball of radius $R$.

\begin{theo}\label{t_lln}
For all $\lambda>0$ there exists a unique Borel probability measure $\nu_\lambda$ on $\Omega$ such that
for any bounded local function $f$, for $\mathbb Q$ almost all $\omega$ and $P$ almost surely we have
$$
\lim\limits_{t\to\infty}\ \frac{1}{t}\int_0^t f(\omega^\lambda(s))\,ds =\nu_\lambda(f),
$$
where $\omega^\lambda(s)=X_0^{\lambda,\omega}(s).\omega$.
\end{theo}

\begin{theo}\label{t_llnH-1}
For all $\lambda>0$, 
for any  local  $f$ in $H^{-1}_\infty(\Omega)$, then $$
\lim\limits_{t\to\infty}\ \frac{1}{t}\int_0^t f(\omega^\lambda(s))\,ds :=\nu_\lambda(f)
$$
exists for $\mathbb Q$ almost all $\omega$ and $P$ almost surely.

\end{theo}

\begin{cor}\label{corollary}
For all $0<\lambda\leq 1$, the steady state and weak steady state exist.
\end{cor}

\begin{proof}[Proof of Corollary \ref{corollary}]
The existence of the steady state is an immediate consequence of   Theorem \ref{t_lln}.

Let now $f$ belong to ${\tilde H}^{-1}_\infty$. Then there is a bounded  $F$ such that $f=\div F$ and
 $F$ can be approximated by bounded and local functions $F_n$.   Apply Theorem \ref{t_llnH-1} to each $f_n=\div F_n$.
By Lemma \ref{lm:estimateAlambda}, $\vert \nu_\lambda(f_n)-\nu_\lambda(f_m)\vert\leq {\tt C}_1\lambda\, \Vert F_n-F_m\Vert_\infty$.
Therefore the sequence $\nu_\lambda(f_n)$ has a limit, say $a$.

From Lemma \ref{lm:estimateAlambda} (with $p>1$) and Theorem \ref{t_llnH-1}, we deduce that
$\frac 1 t A^{\lambda,\omega}_{0,f_n}(t)$ converges to $\nu_\lambda(f_n)$ in $L^1(\mathcal W,P)$ for $\mathbb Q$ almost all $\omega$.
Applying Lemma \ref{lm:estimateAlambda} again, we see that
$\frac 1 t A^{\lambda,\omega}_{0,f}(t)$ converges to $a$ in $L^1(\mathcal W,P)$ for $\mathbb Q$ almost all $\omega$. In particular,
the limit $a$ does not depend on the choice of $F$ and the approximating sequence $(F_n)$. We call it $\nu_\lambda(f)$.
That $\nu_\lambda$ is a linear continuous functional on ${\tilde H}^{-1}_\infty(\Omega)$ follows at once from Lemma \ref{lm:estimateAlambda}.
\end{proof}

\medskip
The remainder of this section including subsections \ref{ss_reg_time} and \ref{ss_proof-t} is devoted to the proof of Theorems  \ref{t_lln} and \ref{t_llnH-1}

Recall from the proof of Lemma \ref{lm:estimateAlambda}  the notation $z=(x,y)$, $x\in\mathbb R^d$ and $y\in\mathbb R$
and the definition the $\mathbb R^{d+1}$ valued process
$$
Z_z^{\lambda,\omega}(t)=\big(X_x^{\lambda,\omega}(t),\, y+A_{x,f}^{\lambda,\omega}(t)+W^1_t\big),
$$
where $W^1$ is an independent one-dimensional Brownian motion (which is assumed to be defined on the same probability space
$(\mathcal{W},\mathcal{F},P)$ as $W$).

Recall that the generator can be written as
$$
\mathcal{M}^{\lambda,\omega}=(\mathcal{L}^{\lambda,\omega})_x+\frac{1}{2}\partial_y^2+f(x.\omega)\partial_y.
$$
We shall use this formula when $f$ is bounded, as in Theorem \ref{t_lln}.
When $f$ belongs to $H^{-1}_\infty(\Omega)$, we rather use the divergence form (see (\ref{eq:reversible}):
$$
\mathcal{M}^{\lambda,\omega}q =\frac 12\div_x(a^\omega\nabla_xq)+\lambda a^\omega e_1\nabla_xq +
\div_x(F^\omega \partial_y q)-\partial_y(F^\omega\nabla_x q)+\frac 12\partial^2_{y}q,
$$
where $F$ is bounded and satisfies $\div F=f$.
(Observe that although $f$ is local, $F$ need not be local itself.)

\subsection{Regeneration times}\label{ss_reg_time}

We assume that $f$ is local and either $f$ is bounded or $f$ belongs to $H^{-1}_\infty(\Omega)$.

Before embarking in the proofs of Theorems \ref{t_lln} and \ref{t_llnH-1}, let us sketch the main steps of the construction of steady states and explain how this section of the paper is organized.

In both Theorems \ref{t_lln} and \ref{t_llnH-1}, we have to study the convergence of the additive functional
$A^{\lambda,\omega}_{0,f}(t)=\int_0^t f(\omega^\lambda(s))ds$ for $\mathbb Q$ almost all $\omega$.
As in the preceeding paragraph, we shall work with the process $Z^{\lambda,\omega}_x$ and deduce the convergence of
$ A^{\lambda,\omega}_{0,f}(t)$ from the regeneration properties of $Z^{\lambda,\omega}_x$. More precisely, the main idea is to construct an increasing sequence of random times, $\tau_1^\lambda<\tau_2^\lambda<...$, that we call {\it regeneration times} and are such that the increments of the process $Z^{\lambda,\omega}_0$ betweeen successive regeneration times are i.i.d. under the annealed law $\mathbb P_z^\lambda$. Then the convergence of $ \frac1 t A^{\lambda,\omega}_{0,f}(t)$ follows at once from the convergence of
$ \frac1 t Z^{\lambda,\omega}_0(t)$ which in turn follows from the law of large numbers for i.i.d. sequences.

Note that, in order to carry out this programm, we also need some bounds on the moments of the regeneration times. Also observe that, as a useful by-product of this proof, we shall be able to express the limit $\nu_\lambda(f)$ in terms of the increments of the additive functional between two successive regeneration times, see formula \ref{nu_f-def}.

Of course, the decoupling properties along regeneration times is tightly related to {\bf Assumption 4}. Very roughly speaking, we proved in \cite{kn:GMP} that the process $X^{\lambda,\omega}_0$, and therefore also the process $Z^{\lambda,\omega}_0$, is transient in direction $e_1$. So there are non-backtracking times $t$ such that: before $t$, the diffusion only visited the half-space $\{x\,:\, e_1\cdot x< e_1\cdot Z^{\lambda,\omega}_0(t)\}$ and after time $t$, it will only visit the half-space $\{x\,:\, e_1\cdot x\geq e_1\cdot Z^{\lambda,\omega}_0(t)\}$. And since, due to {\bf Assumption 4}, the restrictions of the environment in these two half-spaces are independent, we are done.
This is obviously wrong for at least two reasons. First a diffusion process never does such a thing as non-backtracking. Secondly, in order to use {\bf Assumption 4}, we need a little bit of space between the two hyperplanes.
Let us discuss how these issues are addressed in \cite{kn:LS}.

 We carry the whole construction on path space, equipped with the annealed law.
 As a first step towards the desired decoupling property we enlarge the path space with the addition of a sequence of independent Bernoulli random variables $Y_k$ and provide a coupling of this sequence and the diffusion $Z^{\lambda,\omega}_0$,
 see Proposition \ref{p_reg_ti}. The coupling is constructed such that, at times where a Bernoulli variable $Y_k$ takes the value $1$, the canonical process temporarily forgets about the environment and makes a `deterministic' jump in direction $e_1$ of size $9R(\lambda)$. (Here `deterministic' means `independently of what the environment may look like' and $R(\lambda)$ is a parameter that will be chosen later.) If $Y_k$ takes the value $0$, we just do what should be done to retrieve the law of
 $Z^{\lambda,\omega}_0$. Of course, we should tune the parameter $\delta$ of the Bernoulli variables $Y$ so that this `deterministic' jump we impose has a positive probability to occur. How to choose $\delta$ then depends on a lower bound of the transition kernel of the process $Z^{\lambda,\omega}_0$, see Lemma \ref{l_lbt} and note that we need a lower bound that is uniform with respect to $\omega$ and that, since the process $Z^{\lambda,\omega}_0$ involves $f$, the best value we can use for $\delta$ depends on $f$. It also depends on $\lambda$.

Regeneration times will then be times $t$ such that (1) at time $t$, the process reaches a local maximum (with a variation of $R(\lambda)$) and the corresponding Bernoulli variable  in the sequence $Y_k$ takes the value $1$ and (2) after time $t$, the process does not backtrack more than $R(\lambda)$. This construction allows one to explicitly express how the process depends on the restrictions of the environment in the two half-spaces already discussed (and now separated by a distance $R(\lambda)$).
$R(\lambda)$ is chosen larger than the range of dependence $R$ from {\bf Assumption 4} and the size of the support of the local function $f$: $R_f$.

The organization of the rest of this section is as follows.
After introducing some notation on path space, we state the bounds we shall need to choose the parameter $\delta$. Proposition \ref{p_reg_ti} is borrowed from \cite{kn:LS}; it describes the properties of the coupling construction of the diffusion and the Bernoulli random variables. Then we give a detailed definition of the regeneration times starting with formula (\ref{Mdef}).  Theorem \ref{thm:renew}, also borrowed from  \cite{kn:LS}, says that the increments of the process between successive regeneration times are indeed i.i.d. We do not prove Theorem \ref{thm:renew} but we do include the proof of Lemma \ref{l_bou_rt}, that we shall need later and which is actually very close to the proof of Theorem \ref{thm:renew}. In Proposition \ref{p_finitreg}, we establish some bounds on the regeneration times. Finally, in sub-section \ref{ss_proof-t}, we finish the proofs of Theorems \ref{t_lln} and \ref{t_llnH-1}.

The construction of regeneration times will also be used in the proof of fluctuation-dissipation relations and then we shall need bounds on the regeneration times that depend on $\lambda$, see sub-section \ref{ss_est_reg_t} .


The regeneration times will be constructed on canonical space $C([0,\infty), \mathbb R^{d+1})$.  We use the notation $Z(t)_{t\geq 0}$
for the coordinate map on $C([0,\infty), \mathbb R^{d+1})$. The first $d$ components of $Z(\cdot)$ will be denoted by $X(\cdot)$.
Let $P_z^{\lambda,\omega}$  be the law of $Z_z^{\lambda,\omega}$, and $E_z^{\lambda,\omega}$ be the corresponding
expectation.  Let $\mathbb P_z^\lambda$ be the annealed law
$$
\mathbb P_z^\lambda(A)=\int d\mathbb Q(\omega)\int dP_z^{\lambda,\omega}(w){\bf 1}_A(\omega,w)
$$
for any measurable subset $A\subset\Omega\times C([0,\infty),\mathbb R^{d+1})$.

\medskip
Next we set
$$
R(\lambda)= \max\big\{R, R_f, \frac{1}{\lambda}\big\},
$$
where $R$ is the constant from {\bf Assumption 4} and $R_f$ is chosen so that $f$ is measurable with respect to the $\sigma$-field
$\mathcal H_{B_{R_f}}$.  Denoting $B_r(z)$ the ball in $\mathbb R^{d+1}$ centered at $z$ of radius $r$, we let
$U^z=B_{6R(\lambda)}(z+5R(\lambda)\check e_1)$,    $B^z=B_{R(\lambda)}(z+9R(\lambda)\check e_1)$ with
$\check e_1=(e_1,0)$. Then we set
\begin{equation}\label{exi_tim}
T_{U^z}=\inf \{s\geq 0\,:\, Z(s)\not\in U^z\}
\end{equation}
so that $T_{U^z}$ is the exit time from $U^z$.  We also define the corresponding transition densities $p_{\lambda,\omega, U^z}(s,z',z'')$
which satisfy the relation
$$
P_{z'}^{\lambda,\omega}\big\{Z(s)\in G,\, T_{U^{z'}}>s\}=\int _G p_{\lambda,\omega, U^z}(s,z',z'')\,dz''
$$
for any Borel set $G\subset U^{z'}$.

\begin{lm}\label{l_lbt}
Let $0<\lambda\leq 1$.
There exists $\delta_f^\lambda>0$ such that
\begin{equation}\label{bou_dens}
p_{\lambda,\omega, U^z}(\lambda^{-2},z',z'')\geq \frac{2\delta_f^\lambda}{|B_{R(\lambda)}|}
\end{equation}
 for all $z'\in\mathbb R^{d+1}$, $z''\in B^z$. Moreover
 for any $R_0$ there exists $\delta^\lambda>0$ such that
 we may choose $\delta_f^\lambda\geq\delta^\lambda$ for any $f$ such that $R_f\leq R_0$ and either $|f|\leq 1$ or $\Vert f\Vert_{H^{-1}_\infty(\Omega)}\leq 1$.
\end{lm}
\begin{proof}
The required bound is a consequence of the fundamental solution estimates
obtained in \cite{kn:aronson}, see Lemma 5.2 in \cite{kn:GMP}.  Remember that due to {\bf Assumption 2} the matrix $a$ is uniformly elliptic and
either $f$ is bounded or $f=\div F$ where $F$ is bounded.
\end{proof}

We proceed with introducing a coupling construction. We mostly follow the construction  of \cite{kn:LS} (see also \cite{kn:GMP}). First,
we enlarge the probability space by adding a sequence $\{Y_k\}_{k=0}^\infty$ of i.i.d. Bernoulli random variables.  Let $(\mathcal{F}_t)_{t\geq 0}$
be the filtration generated by $(Z(t))_{t\geq0}$ and $\mathcal{J}_m=\sigma\{Y_0,\ldots,Y_m\}$.  Let $\theta_m^\lambda$ be the rescaled shift operator
defined by
$$
\theta_m^\lambda(Z(\cdot))(s)=Z(\lambda^{-2}m+s),\quad s\geq 0.
$$
We extend these operators by setting
$$
\theta_m^\lambda((Z(s))_{s\geq0},\,(Y_k)_{k\geq0})=((Z(\lambda^{-2}m+s))_{s\geq0},\,(Y_{k+m})_{k\geq0}), \quad m\in\mathbb N.
$$

Part (i) of the Proposition below states that we indeed couple i.i.d. Bernoulli random variables and $P^{\l, \o}_z $. Part (ii) expresses the Markov property of the coupling $\widehat{P}^{\l, \o}_z $. Part (iii) says that, when a variable $Y_k$ takes the value $1$, then the diffusion makes this `deterministic' jump we discussed in the introduction of this section.

\begin{prop} \label{p_reg_ti}
There exists, for every $\l$, $\o$ and $z$, a probability measure $\widehat{P}^{\l, \o}_z $ on the enlarged probability space
such that, with $\delta=\delta_f^\lambda$ from Lemma \ref{l_lbt},
\begin{itemize}
\item
[(i)] The law of $(Z(t))_{t \geq 0}$
under $\widehat P^{\l, \o}_z $ is $P^{\l, \o}_z $, and the sequence $(Y_k)_{k \geq 0}$ is
a sequence of i.i.d. Bernoulli variables with success probability $\delta$ under $\widehat{P}^{\l, \o}_z $.
\item
[(ii)] Under $\widehat P^{\l, \o}_z $,
$(Y_n)_{n \geq m}$
is independent of $\FF_{\l^{-2}m} \times {\mathcal J}_{m-1}$,
and conditioned on
$\FF_{\l^{-2}m} \times {\mathcal J}_m$, \ $Z\circ \theta^\l_m$ has the same law as $Z$ under
$\widehat{P}^{\l, \o}_{Z(\l^{-2}m),Y_m} $, where
$\widehat{P}^{\l, \o}_{z,y} $ denotes the conditioned law $\widehat{P}^{\l, \o}_z[\cdot | Y_0 =y]$, (for $y \in \{0,1\}$).
\item
[(iii)] $\widehat{P}^{\l, \o}_{z,1} $-almost surely, $Z(t) \in U^z$ for $t \in [0,\l^{-2}]$ and
the distribution of $Z(\l^{-2})$ under $\widehat{P}^{\l, \o}_{z,1} $ is the uniform distribution on $B^z$.
\end{itemize}
\end{prop}
We refer to \cite{kn:LS} for the proof.

\begin{rmk}\label{r_mark-pro}
As a consequence of Proposition \ref{p_reg_ti},
under  $\widehat P^{\l, \o}_z $,
conditioned on
$\FF_{\l^{-2}m} \times {\mathcal J}_{(m-1)}$, \ $Z\circ \theta^\l_m$ has the same law as $Z$ under
$\widehat{P}^{\l, \o}_{Z(\l^{-2}m)} $. 
\end{rmk}

\bigskip

We will now provide the construction of the sequence of regeneration times.
This construction is algorithmic i.e., in the next paragraph, we describe an algorithm than eventually stops
after a certain number of steps, here denoted with $K$, and delivers the value of the first regeneration time $\tau_1^\lambda$.
The algorithm depends on the choice of the parameter $a$ that we set equal to $3\lambda R(\lambda)$. Its input is a trajectory.

First we introduce a sequence of random times $V^\lambda_k(a)$ when the process $e_1\cdot X(s)$ reaches a local maximum
within a variation of $R(\l)$. If we sampled the trajectory at the times $V^\lambda_k(a)$, we would see increments of order $\lambda^{-1}$ in the direction $e_1$.

Because the coupling in Proposition \ref{p_reg_ti} uses discrete times, we modify the times $V^\lambda_k(a)$ by taking their integer part and thus define the sequence $\tilde{N}^\lambda_k(a)$. From the sequence $\tilde{N}^\lambda_k(a)$, we extract
$N^\lambda_1(a)$ for which the corresponding random variables $Y_k$ takes the value $1$ for the first time. Remember that, when a Bernoulli variable $Y_k$ takes the value $1$, then the diffusion performs a `deterministic' jump of size $\lambda^{-1}$ in direction $e_1$ and in time $\lambda^{-2}$. We look at the process right after the jump.

At time $S_1^\lambda=N_1^\lambda(a)+\lambda^{-2}$, we ask wether the diffusion is going to backtrack in direction $e_1$ by a distance larger than $R(\lambda)$. If the answer is `yes', we then wait until the diffusion backtracks - this defines the backtracking time $D$ - and start the algorithm again: we then get a second random time $S_2^\lambda$; ask if the diffusion will backtrack again ...
The algorithm stops the first time we reach a time $S_k^\lambda$ after which the diffusion does not backtrack more than $R(\lambda)$.
The following definitions provide a rigourous description of the algorithm, including some further technical restrictions.

 Observe that the times $S_k^\lambda$ are stopping times. However, because it includes a non-backtracking condition, the regeneration time $\tau_1^\lambda$ itself is not a stopping time.


%

Let
\beq\label{Mdef}
M(t): = \sup\{e_1\cdot( X(s)-  X(0)): 0 \leq s \leq t\}.
\eeq
For $a> 0$, define the stopping times $V_k^\l(a), k \geq 1$, as follows. We define $T_L=\inf\{t\,: \, \tr{e_1}\,(X(t)-X(0))=L\}$, and define
\beq
V_0^\l(a) : =  T_{\l^{-1}a}, \quad V_{k+1}^\l(a) : = T_{M(\lceil V_k^\l(a)\rceil_\l)+ R(\l)}, \quad k \geq 1;
\eeq
here and later on $\lceil r\rceil_\l$ stands for the $\min\{n\in\mathbb \l^{-2}\mathbb Z\,:\, r\le n\}$.
Then define
\beq
\widetilde N_1^\l(a): =\inf\left\{\big\lceil V_k^\l(a)\big\rceil_\l: k\geq 0, \sup\limits_{s \in [V_k^\l(a), \lceil V_k^\l(a) \rceil_\l ]}\left|e_1\! \cdot \Big( X(s) - X\big(V^\l_k(a))\Big)\right| \leq \frac{ R(\l)}{2}\right\},
\eeq
\beq
\widetilde N^\l_{k+1}(a): = \widetilde N_1^\l(3\l R(\l))\circ \theta^\l_{\l^2\widetilde N_k^\l(a)} + \widetilde N_k^\l(a), \quad k\geq 1\, ,
\eeq
\beq
N_1^\l(a) : = \inf\left\{\widetilde N_k^\l(a): k \geq 1, Y_{\l^2\widetilde N_k^\l(a)} =1\right\}\, ,
\eeq
(we will see later that $\widetilde N_k^\l(a) < \infty$, for all $k$). The random times $\l^2\widetilde N_k^\l(a)$ are integer-valued and $\sup\limits_{s\leq \widetilde N_k^\l(a)}e_1\cdot (X(s) - X(\widetilde N_k^\l(a))) \leq R(\l)$.
We next define random times $S^\lambda_1, J^\lambda_1$ and $R^\lambda_1$ as follows.
\beq
S_1^\l: = N_1^\l(3\l R(\l)) +\l^{-2}, \quad J_1^\l : = S_1^\l+ T_{- R(\l)} \circ \theta^\l_{\l^2 S_1^\l}, \quad R_1^\l : = \lceil J_1^\l\rceil_\l = S_1^\l + D\circ \theta^\l_{\l^2S_1^\l} ,
\eeq
where
\beq\label{Ddef}
D: = \lceil T_{-R(\l)}\rceil_\l\,  .
\eeq
Now we proceed recursively:
\beq
N_{k+1}^\l = R_k^\l + N_1^\l(a_k)\circ \theta^\l_{\l^2R_k^\l} \quad \hbox{ with } a_k = \l\big(M(R_k^\l) - e_1 \cdot(X(R_k^\l)- X(0)) +R(\l)\big)
\eeq
and
$$
S_{k+1}^\l: = N_{k+1}^\l +\l^{-2}, \quad J^\l_{k+1} : = S^\l_{k+1}+ T_{R(\l)} \circ\theta^\l_{\l^2S^\l_{k+1}} ,
\quad R_{k+1}^\l : = \lceil J_{k+1}^\l\rceil_\l = S_{k+1}^\l + D\circ \theta^\l_{\l^2S^\l_{k+1}} \, .
$$
Note that for all $k$, the $\FF_t \times {\mathcal S}_{\l^2\lceil t\rceil_\l}$- stopping times $\l^2N^\l_k,\ \l^2S^\l_k$ and $\l^2R^\l_k$ are integer-valued (the value $+\infty$ is possible). By definition, we have $\l^{-2} \leq N^\l_1 \leq S^\l_1 \leq J^\l_1\leq R^\l_1 \leq N^\l_2 \leq S^\l_2 \leq J^\l_2 \leq R^\l_2 \leq N^\l_3 \ldots \leq \infty$.
The first regeneration time $\tau^\lambda_1$ is defined as
\beq
\tau^\lambda_1 : = \inf\{S_k^\l: S_k^\l < \infty, \, \, R_k^\l = \infty \} \leq \infty\, .
\eeq
Let
\begin{equation}\label{K_def}
K=\inf \{k\geq 1\,:\, S_k^\lambda<\infty \hbox{ and } R_k^\lambda=\infty\}.
\end{equation}
Then $\tau^\lambda_1=S_K^\lambda$.
By definition, $\l^2\tau^\lambda_1$ is integer-valued and $\tau^\lambda_1 \geq 2\l^{-2}$ (since $N_1^\l \geq \l^{-2}$). We see that on the
event $\tau^\lambda_1 < \infty$ it holds
$$
e_1\cdot X(s) \leq e_1\cdot X(\tau^\lambda_1-\l^{-2})+ R(\l)\leq  e_1\cdot X(\tau^\lambda_1) - 7R(\l),\quad \hbox{for }s \leq \tau^\lambda_1 -\l^{-2}, \ \ \widehat P^{\l,\o}_z-\hbox{a.s.},
$$
see also Proposition \ref{p_reg_ti}, \ i.e. $(Z(s))_{s\leq \tau_1 -\l^{-2}}$ remains in the half-space
$\{z \in \R^{d+1}: \check e_1\cdot z \leq  \check e_1\cdot Z(\tau^\lambda_1) -7 R(\l)\}$. On the other hand, since the process $(e_1\cdot X(t))_{t \geq 0}$ never goes below $e_1 \cdot X(\tau^\lambda_1) - R(\l)$ after $\tau^\lambda_1$, $\widehat P^{\l,\o}_z$-a.s., $( Z(t))_{t > \tau^\lambda_1}$ remains in the half-space
$\{z\in \R^{d+1}: \check e_1 \cdot z \geq \check e_1\cdot Z(\tau^\lambda_1) - R(\l)\}$.

Let us define the annealed law
\beq\label{eq:annhat}
\widehat\P^{\l}_z[A]: =\int d\Q(\o)\int d\widehat P^{\l,\o}_z(w)\1_A(\o,w)\, .
\eeq
It has been proved in  \cite{kn:LS} (see also Proposition 5.5 in \cite{kn:GMP}  that $\tau^\lambda_1 < \infty$  $\widehat\P^{\l}_0$-a.s.

For $k\geq 2$ we recursively define
$$
\tau_k^\lambda =\tau_{k-1}^\lambda+ \tau_1^\lambda\circ\theta^\lambda_{\lambda^2\tau^\lambda_{k-1}}.
$$
Then $\tau_k^\lambda$ is finite $\widehat\P^{\l}_0$-a.s. for all $k$.
We set $\tau^\lambda_0=0$ for convenience.

The next theorem is Theorem 2.5 in \cite{kn:LS}.
\begin{theo}\label{thm:renew}
Under the measure $\widehat\P^{\l}_0$, the random variables \\
$\left(\left(Z(\tau^\lambda_{k +1})- Z(\tau^\lambda_k), \tau^\lambda_{k+1} - \tau^\lambda_k \right), k \geq 0 \right)$ are independent;  furthermore, for $k\geq 1$ they  are i.i.d.
 and have the same law as $\left(
Z(\tau^\lambda_{1}), \tau^\lambda_{1} \right)$ under $\widehat\P^{\l}_0[\,  \cdot\, | D=\infty]$.
\end{theo}

Furthermore, we have the following
\begin{lm}\label{l_bou_rt}
Let $(H(m))_{m\geq 0}$ be a random process such that $H(m)$ is measurable with respect to $\mathcal{F}_{\lambda^{-2}m}\times\mathcal{J}_{m-1}$
for all $m$ and such that $\widehat{\mathbb E}^\lambda_0[|H(\lambda^2\tau^\lambda_1)|]<\infty$. Then
$$
\widehat{\mathbb E}^\lambda_0[H(\lambda^2\tau^\lambda_1)|D=\infty]=\sum_{k=1}^\infty \widehat{\mathbb E}^\lambda_0[H(\lambda^2 S_k^\lambda) \mathbf{1}_{\{S_k^\lambda<D\}}].
$$
\end{lm}
\begin{proof}
$$
\widehat{\mathbb E}_0^\lambda\big([H(\lambda^2\tau^\lambda_1)]{\bf 1}_{\{D=\infty\}}\big)
$$
$$
=\sum\limits_{k=1}^\infty\int
\widehat{E}_0^{\lambda,\omega}\big( [H(\lambda^2 S^\lambda_k)]){\bf 1}_{\{S^\lambda_k<\infty\}}{\bf 1}_{\{D\circ\theta^\lambda_{S^\lambda_k}=\infty\}}{\bf 1}_{\{D=\infty\}}\big)\,d\mathbb Q
$$
$$
=\sum\limits_{k=1}^\infty\int
\widehat{E}_0^{\lambda,\omega}\big( [H(\lambda^2 S^\lambda_k)]){\bf 1}_{\{S^\lambda_k<D\}}{\bf 1}_{\{D\circ\theta^\lambda_{S^\lambda_k}=\infty\}}\big)\,d\mathbb Q
$$
$$
=\sum\limits_{k=1}^\infty\int
\widehat{E}_0^{\lambda,\omega}\big( [H(\lambda^2S^\lambda_k)]
{\bf 1}_{\{S^\lambda_k<D\}}\widehat{E}_{Z(S_k^\lambda)}^{\lambda,\omega}{\bf 1}_{\{D=\infty\}}\big)\,d\mathbb Q;
$$
here, to justify the second equality, we have used the fact that if $S_k^\lambda<D$ and $D\circ\theta^\lambda_{S^\lambda_k}=\infty$, then $D=\infty$.
To justify the last equality we have used the fact that $\lambda^2S_k^\lambda{\bf 1}_{\{S_k^\lambda<D\}}$ is a stopping time with respect to the filtration
 $(\FF_{\l^{-2}m} \times {\mathcal J}_{(m-1)}\,,\,m\geq 0)$ and Remark \ref{r_mark-pro}.

For  given $\omega$ and $k$, let $\rho_k^{\lambda,\omega}$ be the law of $Z(S^\lambda_k)$ under $\widehat{P}^{\lambda,\omega}_0$.
Then
$$
\int\big(
\widehat{E}_0^{\lambda,\omega}\big( [H(\lambda^2S^\lambda_k)]
{\bf 1}_{\{S^\lambda_k<D\}}\widehat{E}_{Z(S_k^\lambda)}^{\lambda,\omega}{\bf 1}_{\{D=\infty\}}\big)\,d\mathbb Q
$$
$$
=
\int\Big(\widehat{E}^{\lambda,\omega}_0\left\{\widehat{E}^{\lambda,\omega}_0\left[
H(\lambda^2 S_k^\lambda){\bf 1}_{\{S_k^\lambda<D\}}\big|Z(S_k^\lambda)\right]
\widehat{E}^{\lambda,\omega}_{Z(S_k^\lambda)}\left[
{\bf 1}_{\{D=\infty\}}\right]\widehat{E}_{Z(S_k^\lambda)}^{\lambda,\omega}
\right\}\Big)\,d\mathbb Q
$$
$$
=
\int_\Omega\Big(\int_{\mathbb R^d} \rho_k^{\lambda,\omega}(dz)
\widehat{E}^{\lambda,\omega}_0\left[
H(\lambda^2 S_k^\lambda){\bf 1}_{\{S_k^\lambda<D\}}\big|Z(S_k^\lambda)=z\right]{E}_{z}
^{\lambda,\omega}{\bf 1}_{\{D=\infty\}}
\Big)\,d\mathbb Q
$$
$$
=
\int_{\mathbb R^d} \int_\Omega\Big(\rho_k^{\lambda,\omega}(dz)
\widehat{E}^{\lambda,\omega}_0\left[
H(\lambda^2 S_k^\lambda){\bf 1}_{\{S_k^\lambda<D\}}\big|Z(S_k^\lambda)=z\right]{E}_{z}
^{\lambda,\omega}{\bf 1}_{\{D=\infty\}}
\Big)\,d\mathbb Q.
$$
By the definition of $D$ and $S_k^\lambda$, the term ${E}_{z}
^{\lambda,\omega}{\bf 1}_{\{D=\infty\}}$  is measurable with respect to the $\sigma$-field generated by $\{\sigma(z'\cdot\omega)\,:\, z'\cdot e_1\geq z\cdot e_1-R(\lambda)\}$, and $ \rho_k^{\lambda,\omega}(dz)
\widehat{E}^{\lambda,\omega}_0\left[
H(\lambda^2 S_k^\lambda){\bf 1}_{\{S_k^\lambda<D\}}\big|Z(S_k^\lambda)=z\right]$ is  measurable with respect to the $\sigma$-field generated by $\{\sigma(z'\cdot\omega)\,:\, z'\cdot e_1\leq z\cdot e_1-8R(\lambda)\}$.
Due to {\bf Assumption 4}, these two terms are independent. Therefore,
$$
\int_{\mathbb R^d} \int_\Omega\Big(\rho_k^{\lambda,\omega}(dz)
\widehat{E}^{\lambda,\omega}_0\left[
H(\lambda^2 S_k^\lambda){\bf 1}_{\{S_k^\lambda<D\}}\big|Z(S_k^\lambda)=z\right]{E}_{z}
^{\lambda,\omega}{\bf 1}_{\{D=\infty\}}
\Big)\,d\mathbb Q
$$
$$
=
\int_{\mathbb R^d} \int_\Omega\Big(\rho_k^{\lambda,\omega}(dz)
\widehat{E}^{\lambda,\omega}_0\left[
H(\lambda^2 S_k^\lambda){\bf 1}_{\{S_k^\lambda<D\}}\big|Z(S_k^\lambda)=z\right]\Big)\,d\mathbb Q\ \int_\Omega\Big({E}_{z}
^{\lambda,\omega}{\bf 1}_{\{D=\infty\}}
\Big)\,d\mathbb Q.
$$
The term  $\mathbb E{E}_{z}
^{\lambda,\omega}{\bf 1}_{\{D=\infty\}}$ does not depend on $z$ and equals $\widehat{\mathbb P}^\lambda_0(D=\infty)$.
Thus the last term in the previous formula is equal to
$$
\widehat{\mathbb P}^\lambda_0(D=\infty)
\int \int\Big(\rho_k^{\lambda,\omega}(dz)
\widehat{E}^{\lambda,\omega}_0\left[
H(\lambda^2 S_k^\lambda){\bf 1}_{\{S_k^\lambda<D\}}\big|Z(S_k^\lambda)=z\right]\Big)\,d\mathbb Q
$$
$$
= \widehat{\mathbb P}^\lambda_0(D=\infty) \widehat{\mathbb E}^\lambda_0\big(H(\lambda^2 S_k^\lambda){\bf 1}_{\{S_k^\lambda<D\}}\big),
$$
which implies the desired relation.
\end{proof}

The next statement provides us with useful estimates for the regeneration times.
\begin{prop}\label{p_finitreg}
There exists a constant $C^\lambda_f>0$ such that
$$
\widehat{\mathbb E}^\lambda_0\big[e^{C^\lambda_f\tau^\lambda_1}\big]<\infty
\quad
\hbox{\rm and}\quad
\widehat{\mathbb E}^\lambda_0\big[e^{C^\lambda_f (e_1\cdot X(\tau_1^\lambda))}\big]<\infty.
$$
Moreover  for any $R_0$ there exists $C^\lambda>0$ such that
 we may choose $C_f^\lambda\geq C^\lambda$ for any $f$ such that $R_f\leq R_0$ and either $|f|\leq 1$ or $\Vert f\Vert_{H^{-1}_\infty(\Omega)}\leq 1$.
\end{prop}
\begin{proof}
The first claim of the Proposition is proved
in \cite{kn:LS}, Theorem 4.9 and Corollary 4.10.
As for the second claim observe from the construction of $\tau_1^\lambda$ that, once $R(\lambda)$ is chosen, and given the $Y_k$'s, the definition
of $\tau_1^\lambda$ only involves the process $e_1\cdot X$.  Therefore, the rate of decay of the distribution function of $\tau_1^\lambda$ depends on
$f$ only through $R_f$ and $\delta_f^\lambda$. Besides, the bigger $\delta_f^\lambda$, the faster this distribution function decays.
We conclude the proof with the second claim of Lemma \ref{l_lbt}.
\end{proof}

\subsection{Proof of Theorems \ref{t_lln} and  \ref{t_llnH-1}     }\label{ss_proof-t}

The law of $(Z(t))_{t \geq 0}$
under $\widehat P^{\l, \o}_0 $ is the law of $(Z_0^{\lambda,\omega}(t))_{t \geq 0}$ under $P$.
Therefore, under $\widehat P^{\l, \o}_0 $,  the last component of  $Z(\cdot)$ is a semimartingale of the form ${\tt W}^1_\cdot+{\tt A}_f(\cdot)$
where  ${\tt W}_\cdot^1$ is a Brownian motion and  the  law of ${\tt A}_f$ is the law of $A^{\lambda,\omega}_{0,f}$ under
$P$.

It follows from Theorem \ref{thm:renew} and Proposition \ref{p_finitreg} that
$$
\frac{1}{k}\tau^\lambda_k\longrightarrow \widehat{\mathbb E}_0^\lambda\big[\tau_1^\lambda|D=\infty\big] \ \ \hbox{as }k\to\infty\qquad
 \widehat{\mathbb P}_0^\lambda-{\textrm a.s.}
$$
and
$$
\frac{1}{k}\Big({\tt A}_f(\tau^\lambda_k)+{\tt W}^1_{\tau^\lambda_k}\Big)\longrightarrow
\widehat{\mathbb E}_0^\lambda\big[{\tt A}_f(\tau^\lambda_1)+{\tt W}^1_{\tau^\lambda_1}|D=\infty\big] \ \ \hbox{as }k\to\infty\qquad
 \widehat{\mathbb P}_0^\lambda-{\textrm a.s.}
$$
Since $k^{-1}{\tt W}^1_{\tau^\lambda_k}$ a.s. converges to zero, we derive from the previous relation that
\begin{equation}\label{eff_velo}
\frac{{\tt A}_f(\tau^\lambda_k)}{\tau_k^\lambda}\longrightarrow
\frac{\widehat{\mathbb E}_0^\lambda\big[{\tt A}_f(\tau^\lambda_1)+{\tt W}^1_{\tau^\lambda_1}|D=\infty\big]}{\widehat{\mathbb E}_0^\lambda\big[\tau^\lambda_1|D=\infty
\big] }\ \ \hbox{as }k\to\infty\qquad
 \widehat{\mathbb P}_0^\lambda-{\textrm a.s.}
\end{equation}
Let us show that the term $\widehat{\mathbb E}_0^\lambda\big[{\tt W}^1_{\tau^\lambda_1}|D=\infty\big]$
on the right-hand side of (\ref{eff_velo}) vanishes.  Since $Z$ and ${\tt A}_f$ are additive functionals of $Z$, then
 ${\tt W}^1$ is also an additive functional of $Z$. From the Markov property of ${P}^{\l, \o}_{Z(\l^{-2}m)}$, we get that the process $({\tt W}^1_{\lambda^{-2}m+t}-{\tt W}^1_{\lambda^{-2}m})_{t\geq 0}$  is a Brownian motion independent of  $\FF_{\l^{-2}m} \times {\mathcal J}_{(m-1)}$.
 Since ${\bf 1}_{\{S_k^\lambda<D\}}\lambda^2S_k^\lambda$ is a stopping time with respect to the filtration $(\FF_{\l^{-2}m} \times {\mathcal J}_{(m-1)})_{m\geq 0}$, we have
 $\widehat{\mathbb E}_0^\lambda[{\tt W}^1_{S_k^\lambda}{\bf 1}_{\{S_k^\lambda<D\}}]=0$ for all $k$.
 Combining this with Lemma \ref{l_bou_rt} yields that $\widehat{\mathbb E}_0^\lambda\big[{\tt W}^1_{\tau^\lambda_1}|D=\infty\big]$ vanishes.
Therefore, (\ref{eff_velo}) takes the form
\begin{equation}\label{eff_velo_mod}
\frac{{\tt A}_f(\tau^\lambda_k)}{\tau_k^\lambda}\longrightarrow
\frac{\widehat{\mathbb E}_0^\lambda\big[{\tt A}_f(\tau^\lambda_1)|D=\infty\big]}{\widehat{\mathbb E}_0^\lambda\big[\tau^\lambda_1|D=\infty
\big] }\ \ \hbox{as }k\to\infty\qquad
 \widehat{\mathbb P}_0^\lambda-{\textrm a.s.}
\end{equation}
We introduce the notation
\begin{equation}\label{nu_f-def}
\nu_\lambda(f)=\frac{\widehat{\mathbb E}_0^\lambda\big[{\tt A}_f(\tau^\lambda_1)|D=\infty\big]}{\widehat{\mathbb E}_0^\lambda\big[\tau^\lambda_1|D=\infty
\big] }= \frac{\widehat{\mathbb E}_0^\lambda\big[{\tt A}_f(\tau^\lambda_1)+{\tt W}^1_{\tau^\lambda_1}|D=\infty\big]}{\widehat{\mathbb E}_0^\lambda\big[\tau^\lambda_1|D=\infty
\big] }.
\end{equation}
Using  standard arguments based on Proposition \ref{p_finitreg} we can replace the limit along the sequence $\{\tau_k^\lambda\}$  in (\ref{eff_velo_mod})
with the limit with respect to $t$. Therefore, we conclude that $t^{-1} {\tt A}_f(t) $ a.s. converges to $\nu_\lambda(f)$.
This implies that $t^{-1}A_{0,f}^{\lambda,\omega}$ also converges to  $\nu_\lambda(f)$ as $t\to\infty$ for $\mathbb Q$ almost all $\omega$
and $P$-a.s.
This yields the statement of Theorem \ref{t_llnH-1}.

To complete the proof of Theorem \ref{t_lln}, it remains to show that $\nu_\lambda$ is a Borel probability measure on
$\Omega$.  By construction, $\nu_\lambda$ is a non-negative linear functional on the space of bounded local functions. For any such function $f$
we have $|A_{0,f}^{\lambda,\omega}(f)|\leq t\|f\|_{L^\infty(\Omega)}$. Therefore,  $|\nu_\lambda(f)|\leq\|f\|_{L^\infty(\Omega)}$.
It is obvious that  $\nu_\lambda(1)= 1$. The only property to be justified is the sigma-additivity of $\nu_\lambda$.
Let $R_0>0$ and let $(f_n)_{n\geq 1}$ be a sequence of functions which are measurable with respect to the $\sigma$-field generated by $\{\sigma(y.\omega)\,:\,|y|\leq R_0\}$ and such that $0\leq f_n\leq 1$ and $f_n(\omega)$ tends to zero for all $\omega$.
For all $T>0$ we have
$$
\widehat{\mathbb E}^\lambda_0[{\tt A}_{f_n}(\tau^\lambda_1)]\leq \widehat{\mathbb E}^\lambda_0[{\tt A}_{f_n}(T)]
+\widehat{\mathbb E}^\lambda_0[(\tau^\lambda_1){\bf 1}_{\{\tau^\lambda_1\geq T\}}].
$$
Clearly, for any $T>0$ we have $\widehat{\mathbb E}^\lambda_0[{\tt A}_{f_n}(T)]\to 0$, as $n\to\infty$. Besides, although  the law
of $\tau^\lambda_1$  depends on $f_n$,
due to Proposition \ref{p_finitreg}, $\widehat{\mathbb E}^\lambda_0[(\tau^\lambda_1){\bf 1}_{\{\tau^\lambda_1\geq T\}}]$ tends to
zero, as $T\to\infty$, uniformly in $n$.  This implies that $\widehat{\mathbb E}^\lambda_0[{\tt A}_{f_n}(\tau^\lambda_1)]$
converges to zero, and thus $\nu_\lambda(f_n)$ converges to zero, as $n\to\infty$. Therefore, $\nu_\lambda$ is a  probability Borel measure on the
$\sigma$-field generated by $\{\sigma(y.\omega)\,:\,|y|\leq R_0\}$. And since it holds true for any $R_0$, then $\nu_\lambda$ extends to the whole Borel $\sigma$-field
of $\Omega$.
\qed

\section{Fluctuation dissipation theorem}\label{ss_h-1}

In this Section we compute the derivative of the steady state as $\lambda\to0$. Our main tool is the description of the scaling limit
of regeneration times for small $\lambda$.

Everywhere in this Section {\bf Assumptions 1--4} are fulfilled.

\medskip
Recall the properties of $\bar\Gamma(f)$ from Lemma \ref{l_cova}.

\begin{theo}\label{t_sssl}
Let $f$ be local and belong to $H^{-1}_\infty(\Omega)$ .  Then,
the derivative of $\nu_\lambda(f)$ at $\lambda=0$ exists and is equal to $\bar\Gamma(f)$.
\end{theo}

An immediate corollary of Theorem \ref{t_sssl} and Lemma \ref{lm:estimateAlambda} is the following version
of Theorem \ref{t_ssslIntro}:

\begin{cor} \label{fdt}
Let $f$  belong to ${\tilde H}^{-1}_\infty(\Omega)$ .  Then,
the derivative of $\nu_\lambda(f)$ at $\lambda=0$ exists and is equal to $\bar\Gamma(f)$.
\end{cor}

The proof will be divided into several steps which are detailed in the following subsections.

\subsection{Estimates for regeneration times}\label{ss_est_reg_t}

\begin{prop}\label{p_ert}
Under the conditions of Theorem \ref{t_sssl}
there exist  constants $C_1(f)>0$ and $\overline{C}(f)>0$ such that, for all $\lambda$ with $0< \lambda\leq 1$, we have
$$
\widehat{\mathbb E}^\lambda_0\big[e^{C_1(f)\lambda^2\tau^\lambda_1}\big]\leq\overline{C}(f)
\quad
\hbox{\rm and}\quad
\widehat{\mathbb E}^\lambda_0\big[e^{C_1(f)\lambda (e_1\cdot X(\tau_1^\lambda))}\big]]\leq\overline{C}(f).
$$
\end{prop}
\begin{rmk}
By the same arguments as in the proof of Proposition \ref{p_finitreg}, the constants $C_1(f)$ and $\overline{C}(f)$ can be chosen
to be the same for all functions $f$
such that $R_f\leq R_0$ and $\Vert f\Vert_{H^{-1}_\infty(\Omega)}\leq 1$.
\end{rmk}

\begin{proof}
We will need a version of Lemma \ref{l_lbt} uniform with respect to $\lambda\in(0,1)$.
\begin{lm}\label{l_unif}
Let $f$ be as in Theorem \ref{t_sssl}. Then there exists a constant $\delta_f>0$ such that estimate (\ref{bou_dens}) holds for all $\lambda\in (0,1)$ with $\delta_f^\lambda=\delta_f$.
\end{lm}
\begin{proof}
We recall that the process
$$
Z_z^{\lambda,\omega}(t)=\big(X_x^{\lambda,\omega}(t),\, y+A_{x,f}^{\lambda,\omega}(t)+W^1_t\big)
$$
has a generator in divergence form, see (\ref{eq:reversible}).

In the variables $\tilde z=\lambda z$ and $\tilde t=\lambda^2 t$, this
generator reads, see (\ref{gen_z}):
$$\div_{\tilde x}(a^\omega(\lambda^{-1}\tilde x)\nabla_{\tilde x}q)+ a^\omega((\lambda^{-1}\tilde x)e_1\nabla_{\tilde x}q +
\div_{\tilde x}(F^\omega(\lambda^{-1}\tilde x) \partial_{\tilde y} q)-\partial_{\tilde y}(F^\omega(\lambda^{-1}{\tilde x})\nabla_{\tilde x} q)+ \partial^2_{{\tilde y}}q.
$$
Since $F$ is bounded, then for the corresponding parabolic operator, the Aronson estimates (see \cite{kn:aronson}) hold uniformly in $\lambda$ and in $F$ on any finite time interval and in any fixed  ball and for almost all $\omega$.  This implies that in the  statement of Lemma \ref{l_lbt} we can choose $\delta^\lambda_f$ independent of $\lambda$.
\end{proof}

Turning back to the proof of Proposition \ref{p_ert}, due to Lemma \ref{l_unif}, in the construction of $\tau_1^\lambda$, we can choose
the same Bernoulli random variables $(Y_k)_{k\geq 0}$ for all $\lambda \in(0,1]$. Given the sequence $(Y_k\,:\,k\geq 0)$ and the trajectory $e_1\cdot X(\cdot)$,
the definition of $\tau_1^\lambda$ in Section \ref{ss_reg_time} coincides with the definition of the regeneration time $\tau_1$ in \cite{kn:GMP}  (Notice that
the notation $e_1\cdot X(\cdot)$ and $\widehat{P}^{\lambda,\omega}_0$ are used for the same objects both here and in \cite{kn:GMP}).
We read from Lemma 5.8 and its proof in \cite{kn:GMP} that
$$
\sup\limits_\omega\sup\limits_{0<\lambda\leq 1}\widehat{ E}^{\lambda,\omega}_0\big[e^{C_1(f)\lambda (e_1\cdot X(\tau_1^\lambda))}\big]<\infty,
$$
and
$$
\sup\limits_\omega\sup\limits_{0<\lambda\leq 1}\widehat{ E}^{\lambda,\omega}_0\big[e^{C_1(f)\lambda^2\tau^\lambda_1}\big]<\infty.
$$
These estimates clearly imply the estimates stated in the Proposition.
\end{proof}

\begin{lm}\label{l_transv_b}
Under the conditions of Theorem \ref{t_sssl}
there exist  constants $C_1(f)>0$ and $\overline{C}(f)>0$ such that, for all $\lambda\leq 1$,
\begin{equation}\label{ogo}
\widehat{\mathbb E}^\lambda_0\Big[\exp\Big({C_1(f)\lambda\sup\limits_{0\leq s\leq \tau_1^\lambda} | Z(s)|}\Big)\Big]\leq\overline{C}(f).
\end{equation}
\end{lm}
\begin{proof}
Denote $\widetilde T_r=\inf\{s>0\,:\, |\lambda Z(s/\lambda^2)|=r\}$.  Applying Aronson's lower bound (see \cite[Theorems 8 and 9]{kn:aronson}) to the parabolic equation with  generator given by (\ref{gen_z}), we obtain that there exists $\delta_0>0$ such that for all $\lambda\in(0,1)$ and all $\omega$
$$
P^{\lambda,\omega}_0(\widetilde T_1\geq 1)\geq \delta_0.
$$
Therefore,
$$
E^{\lambda,\omega}_0\big(e^{-\widetilde T_1}\big)\leq 1-\eps_0
$$
for some $\eps_0>0$. Applying the Markov property we deduce that
$$
E^{\lambda,\omega}_0\big(e^{-\widetilde T_r}\big)\leq (1-\eps_0)^{r-1}.
$$
Then
$$
E^{\lambda,\omega}_0\big[\exp\big(c\sup\limits_{0\leq s\leq t}{\lambda|Z(s/\lambda^2)|}\big)\big]=
\int_0^\infty ds\, e^s P^{\lambda,\omega}_0(e^{-\widetilde T_{(s/c)}}\geq e^{-t})$$
$$\leq
e^t\int_0^\infty ds\, e^s (1-\eps_0)^{(s/c)-1}=Ce^t
$$
provided we have chosen $c$ small enough  so that $ e^1(1-\eps_0)^{(1/c)}<1$.
Writing
$$
\widehat{P}^{\lambda,\omega}_0\big(\sup\limits_{0\leq s\leq \tau^\lambda_1}\lambda |Z(s)|\geq T\big)\leq
\widehat{P}^{\lambda,\omega}_0\big(\lambda^2\tau^\lambda_1\geq T/2\big)+
\widehat{P}^{\lambda,\omega}_0\big(\sup\limits_{0\leq s\leq T/2}\lambda |Z(s/\lambda^2)|\geq T\big),
$$
we deduce from Proposition \ref{p_ert} and Lemma \ref{l_unif} that \eqref{ogo} holds true for sufficiently small  $C_1(f)>0$ and some
$\overline{C}(f)>0$.
\end{proof}


\subsection{Scaling limit on regeneration scale}\label{ss_LRbis}


\begin{prop}\label{t_lr_bis}
Under the product measure $P\times \mathbb Q$, the process $\big(\lambda Z^{\lambda,\omega}_0(\lambda^{-2}t)\,;\,
t\geq 0\big)$  converges in law,
in $C([0,\infty), \mathbb R^{d+1})$,
towards a Brownian motion with constant drift.  The limit covariance matrix and the limit drift are given, respectively, by
\begin{equation}\label{ex_lim_c}
\widehat{\Sigma}=\left(
\begin{array}{cc}
\Sigma &0\\[2mm]
0&1+\Sigma(f)
\end{array}
\right)\,,\qquad\widehat{B}=
\left(
\begin{array}{c}
\Sigma e_1\\[2mm]
\overline{\Gamma}(f)
\end{array}
\right)\,.
\end{equation}
\end{prop}

\begin{proof}
From Proposition \ref{p_cov} we get the convergence in law of  $\big(\lambda{X}_0^{\lambda,\omega}(\lambda^{-2}\cdot), \lambda{A}^{\lambda,\omega}_{0,f}(\lambda^{-2}\cdot)\big)$, under the annealed measure $P\times\mathbb Q$. Since $W^1$ is an independent Brownian motion, then the process $\big(\lambda{X}_0^{\lambda,\omega}(\lambda^{-2}\cdot), \lambda{A}^{\lambda,\omega}_{0,f}(\lambda^{-2}\cdot),\,
\lambda{W}^1(\lambda^{-2}\cdot)\big)$ also converges in law.  Since $\big(\lambda{Z}_0^{\lambda,\omega}(\lambda^{-2}\cdot)\big)$ is a linear function of  $\big(\lambda{X}_0^{\lambda,\omega}(\lambda^{-2}\cdot),$  $\lambda{A}^{\lambda,\omega}_{0,f}(\lambda^{-2}\cdot),\,
\lambda{W}^1(\lambda^{-2}\cdot)\big)$, then it also converges in law.

We already computed the limit covariance and drift  in Proposition \ref{p_cov}.
\end{proof}

\subsection{Continuity lemma}\label{ss_tau-proof}
 Let $\mathcal{P}$ be the law of  a Brownian motion with  covariance and drift given by (\ref{ex_lim_c})  on the canonical space $C([0,\infty); \mathbb R^{d+1})$,
 and  let $\mathcal{E}$ be
the corresponding expectation.
In the same way as in  Section \ref{ss_reg_time},
we introduce  the measure $\widehat{\mathcal{P}}$ defined on the extended path space that  includes the sequence of Bernoulli
random variables $(Y_k)_{k\geq 0}$.  Choosing $\lambda=1$, denote
$\check S_k=S^{\lambda=1}_k$, $\check\tau_1=\tau^{\lambda=1}_1$ and the corresponding random variable $\check D$.

Let $\phi=\phi(z,s,\lambda)$, $z\in\mathbb R^{d+1}$, $s\in\mathbb R$, $\lambda\in(0,1)$, be a continuous function  such that
$$
|\phi(z,s,\lambda)|\leq C(1+|z|+|s|)|^m
$$
for some $C>0$ and $m>0$.
\begin{lm}\label{t_fop}
The following continuity relation holds

$$
\lim\limits_{\lambda\to 0}\ \ \frac{\widehat{\mathbb E}_0^\lambda\big(\phi(\lambda Z(\tau_1^\lambda),
\lambda^2\tau_1^\lambda, \lambda)){\bf 1}_{\{D=\infty\}}\big)}{\widehat{\mathbb E}_0^\lambda\big(\lambda^2\tau_1^\lambda
{\bf 1}_{\{D=\infty\}}\big)}
=
\frac{\widehat{\mathcal{ E}}\big(\phi(Z(\check\tau_1),
\check\tau_1, 0)){\bf 1}_{\{\check D=\infty\}}\big)}{\widehat{\mathcal{E}}\big(\check\tau_1
{\bf 1}_{\{\check D=\infty\}}\big)}.
$$
\end{lm}

\begin{proof}
By Lemma \ref{l_bou_rt} with $H(n)=\phi(\lambda Z(\lambda^{-2}n),n,\lambda))$ we get
\begin{equation}\label{tau_deco}
\widehat{\mathbb E}_0^\lambda\big(\phi(\lambda Z(\tau_1^\lambda),
\lambda^2\tau_1^\lambda, \lambda))\big|D=\infty\big)=
\sum\limits_{k=1}^\infty
\widehat{\mathbb E}_0^\lambda\big(\phi(\lambda Z(S_k^\lambda),
\lambda^2 S_k^\lambda, \lambda)) {\bf 1}_{\{S_k^\lambda<D\}} \big).
\end{equation}
For each $k$, the functions $\check S_k$, $Z(\check S_k)$ and ${\bf 1}_{\{\check S_k<D\}}$ are $\widehat{\mathcal{P}}
$-a.s. continuous functions on
path space. By Theorem \ref{t_lr_bis} and the continuity of $\phi$,  then the law of
$\phi(\lambda Z(S_k^\lambda),
\lambda^2 S_k^\lambda, \lambda)) {\bf 1}_{\{S_k^\lambda<D\}} $
under $\widehat{\mathbb P}_0^\lambda$ converges to the law of
$\phi(Z(\check S_k),
\check S_k, 0)) {\bf 1}_{\{\check S_k<\check D\}} $ under $\widehat{\mathcal{ E}}$.
 Combining the inequality $S_k^\lambda{\bf 1}_{\{S_k^\lambda<\infty\}}\leq\tau_1^\lambda$ with Lemma \ref{l_transv_b} and  Proposition \ref{p_ert}, we deduce uniform in $\lambda$ exponential tail estimates for $\lambda^2 S_k^\lambda$ and for $\lambda|Z(S_k^\lambda)|$. Under our standing growth condition on $\phi$, then $\phi(\lambda Z(S_k^\lambda),\lambda^2 S_k^\lambda, \lambda))$ satisfies uniform in $\lambda$ stretched exponential tail estimates. This
 implies that
$$
\widehat{\mathbb E}_0^\lambda\big(\phi(\lambda Z(S_k^\lambda),
\lambda^2 S_k^\lambda, \lambda))\big|D=\infty\big) \longrightarrow
\widehat{\mathcal{ E}}\big(\phi( Z(\check S_k),
\check S_k, 0)) {\bf 1}_{\{\check S_k<\check D\}} \big)
$$
for each $k$.  It remains to bound the tail of the series on the right-hand side of (\ref{tau_deco}). By the Cauchy-Schwartz inequality we have
$$
\left[\widehat{\mathbb E}_0^\lambda\big(\phi(\lambda Z(S_k^\lambda),
\lambda^2 S_k^\lambda, \lambda)) {\bf 1}_{\{S_k^\lambda<D\}} \big)\right]^2\leq\widehat{\mathbb E}_0^\lambda \left[\big(\phi(\lambda Z(S_k^\lambda),
\lambda^2 S_k^\lambda, \lambda)) {\bf 1}_{\{S_k^\lambda<D\}} \big)^2\right] \widehat{\mathbb P}_0^\lambda\big(S_k^\lambda<D \big).
$$
As in the preceding discussion,  Proposition \ref{p_ert} and Lemma \ref{l_transv_b} imply that the first term on the right-hand side is bounded
uniformly in $\lambda$ and $k$. On the other hand,
$$
\widehat{\mathbb P}_0^\lambda\big(S_k^\lambda<D \big)\leq \widehat{\mathbb P}_0^\lambda\big(S_k^\lambda<\infty \big)
\leq\widehat{\mathbb P}_0^\lambda\big(\tau_1^\lambda\geq\lambda^{-2}k \big).
$$
The term on the right-hand side here converges to zero uniformly in $\lambda$ at exponential speed. Therefore, we can pass to the limit
in (\ref{tau_deco}).
\end{proof}

\begin{proof}[Proof of Theorem \ref{t_sssl}]
Denote by $Z_{d+1}(t)$ the $(d+1)$-st component of $Z(t)$.   We read from (\ref{nu_f-def}) that
$$
\frac{1}{\lambda}\nu_\lambda(f)=\frac{\widehat{\mathbb E}_0^\lambda\big[{\tt A}_f(\tau^\lambda_1)+{\tt W}^1_{\tau_1^\lambda}|D=\infty\big]}{\lambda\widehat{\mathbb E}_0^\lambda\big[\tau^\lambda_1|D=\infty
\big] }=\frac{\lambda\widehat{\mathbb E}_0^\lambda\big[Z_{d+1}(\tau_1^\lambda)|D=\infty\big]}{\lambda^2\widehat{\mathbb E}_0^\lambda\big[\tau^\lambda_1|D=\infty
\big] }.
$$
By Lemma \ref{t_fop} we have
$$
\frac{\lambda\widehat{\mathbb E}_0^\lambda\big[Z_{d+1}(\tau_1^\lambda)|D=\infty\big]}{\lambda^2\widehat{\mathbb E}_0^\lambda\big[\tau^\lambda_1|D=\infty\big] }\ \mathop{\longrightarrow}\limits_{\lambda\to0}\  \frac{\widehat{\mathcal{E}}\big[Z_{d+1}(\check \tau_1)|\check D=\infty\big]}{\widehat{\mathcal{E}}\big[\check \tau_1|\check D=\infty
\big] }.
$$
As a special case of  (\ref{nu_f-def}) with a constant $\sigma$ and $\lambda=1$, we know that
$$
\frac{\widehat{\mathcal{E}}\big[Z_{d+1}(\check \tau_1)|\check D=\infty\big]}{\widehat{\mathcal{E}}\big[\check \tau_1|\check D=\infty\big] }=
\lim\limits_{t\to\infty}\frac{Z_{d+1}(t)}{t} \quad \mathcal{P}\hbox{-a.s.}
$$
Obviously, the last limit is equal to $\overline\Gamma(f)$, see (\ref{ex_lim_c}).
\end{proof}

\section{Continuity of variance and Einstein relation}\label{s_sec4}

We assume {\bf Assumptions 1--4} are fulfilled.

\subsection{Einstein relation}\label{ss_einrel}

In this section we obtain the Einstein relation as a consequence of the results of the previous Section.
This proof differs from that given in \cite{kn:GMP}.  We refer to \cite{kn:Einstein} for the original physical intuition.

It follows from Theorem \ref{thm:renew} and Proposition \ref{p_finitreg} taking $f=0$, that for any fixed $\lambda\in(0,1)$,  then $X$ satisfies the law of large numbers under $\widehat{\mathbb P}^{\lambda}_0$. Equivalently, there exists a vector $\ell(\lambda)\in\mathbb R^d$ such that
$$
\lim\limits_{t\to\infty}\frac{1}{t}X^{\lambda,\omega}_0(t)=\ell(\lambda)\qquad
$$
for $\mathbb Q$ almost all $\omega$ and $P$-a.s.

\begin{theo}[Einstein relation]\label{t_einst-rel}
As $\lambda\to 0$, then
$$
\frac{1}{\lambda}\ell(\lambda)\longrightarrow \Sigma e_1.
$$
\end{theo}
\begin{proof}
Using the regeneration structure as in the proof of (\ref{nu_f-def}), we can represent $\ell(\lambda)$ as follows
$$
\ell(\lambda)= \frac{\widehat{\mathbb E}_0^\lambda\big[X(\tau^\lambda_1)|D=\infty\big]}{\widehat{\mathbb E}_0^\lambda\big[\tau^\lambda_1|D=\infty
\big] }.
$$
Therefore,
$$
\frac{1}{\lambda}\ell(\lambda)= \frac{\lambda\widehat{\mathbb E}_0^\lambda\big[X(\tau^\lambda_1)|D=\infty\big]}{\lambda^2\widehat{\mathbb E}_0^\lambda\big[\tau^\lambda_1|D=\infty
\big] }.
$$
It follows from the continuity Lemma \ref{t_fop} that
$$
\frac{1}{\lambda}\ell(\lambda) \longrightarrow \frac{\widehat{\mathcal{E}}\big[X(\check \tau_1)|\check D=\infty\big]}{\widehat{\mathcal{E}}\big[\check \tau_1|\check D=\infty\big] }.
$$
The expression on the right-hand side is the drift of the $X$-components of the process $Z$ under $\mathcal{P}$.  By (\ref{ex_lim_c}), it is equal to $\Sigma e_1$.
\end{proof}

\subsection{Continuity of variance}\label{contvar}

This section deals with  the continuity of the effective variance of $X^{\lambda,\omega}_0$ as $\lambda\to0$.

It follows from Theorem \ref{thm:renew} and Proposition \ref{p_finitreg} taking $f=0$, that for any fixed $\lambda\in(0,1)$,  then $X$ satisfies the central limit theorem under $\widehat{\mathbb P}^{\lambda}_0$: there exists a matrix $\Sigma_\lambda$ such that the law of
$\frac{1}{\sqrt{t}}(X^{\lambda,\omega}_0(t)-\ell(\lambda)t)$ under the annealed measure $P\times\mathbb Q$ converges to the centered Gaussian law
with covariance matrix $\Sigma_\lambda$.

\begin{theo}[Continuity of variance]\label{t_cont-var}
As $\lambda\to 0$, we have
$$
\Sigma_\lambda\longrightarrow \Sigma.
$$
\end{theo}
\begin{proof}
Using the regeneration structure, as in the proof of (\ref{nu_f-def}), we can represent $\Sigma_\lambda$ as follows: for any $e\in\mathbb R^d$ then
$$
e\cdot\Sigma_\lambda e= \frac{\widehat{\mathbb E}_0^\lambda\big[(X(\tau^\lambda_1)\cdot e-\tau_1^\lambda \ell(\lambda)\cdot e)^2|D=\infty\big]}{\widehat{\mathbb E}_0^\lambda\big[\tau^\lambda_1|D=\infty
\big] }=
 \frac{\widehat{\mathbb E}_0^\lambda\big[(\lambda X(\tau^\lambda_1)\cdot e-\lambda^2\tau_1^\lambda \lambda^{-1}\ell(\lambda)\cdot e)^2|D=\infty\big]}{\widehat{\mathbb E}_0^\lambda\big[\lambda^2\tau^\lambda_1|D=\infty
\big] }.
$$
We apply the continuity Lemma \ref{t_fop} to the function $\phi(z,s,\lambda)=(e\cdot x-s\lambda^{-1}\ell(\lambda)\cdot e)^2$ for $\lambda\not=0$,
and  $\phi(z,s,0)=(e\cdot x-s\Sigma e_1\cdot e)^2$. Observe that according to the Theorem \ref{t_einst-rel} (Einstein relation), $\phi$ is continuous.
Then we get
$$
e\cdot\Sigma_\lambda e\longrightarrow
 \frac{\widehat{\mathcal{E}}\big[(X(\check \tau_1)\cdot e-\check\tau_1\Sigma e_1\cdot e)^2|\check D=\infty\big]}{\widehat{\mathcal{E}}\big[\check \tau_1|\check D=\infty\big] }.
$$
The expression on the right-hand side is the diffusion matrix  of the $X$-components of the process $Z$ under $\mathcal{P}$.  By (\ref{ex_lim_c}), it is equal to $\Sigma e\cdot e$.
\end{proof}

\bigskip
It follows from Theorem \ref{thm:renew} and Proposition \ref{p_finitreg}  that for any $f$, a local element in $H^{-1}_\infty(\O)$  and any fixed $\lambda\in(0,1)$, then $A_f$ satisfies the central limit theorem under $\widehat{\mathbb P}^{\lambda}_0$:
there exists a matrix $\Sigma_\lambda(f)$ such that the law of
$\frac{1}{\sqrt{t}}\big(A_f(t)-\nu_\lambda(f)t\big)$ under the annealed measure $P\times\mathbb Q$ converges to the centered Gaussian law
with covariance matrix $\Sigma_\lambda(f)$.

\begin{theo}
As $\lambda\to 0$, we have $\Sigma_\lambda(f)\longrightarrow \Sigma(f)$.
\end{theo}

The proof is the same as above. We leave the details to the reader.

\section{Appendix A}\label{sec:append}

Although our main interest in this paper are diffusions in a random environment, in order to better explain our results, we now briefly discuss the easier case of diffusions in a periodic environment.

In the periodic setting, the role of the dynamics of the environment viewed from the particle is now played by the projection of
the diffusion $X^\lambda_0$ on the torus.

In the case $\lambda=0$, we will get a stationary corrector $\chi_1$, see equation (\ref{eq:percorr}).

When $\lambda\not=0$, the process of the environment seen from the particle has an
absolutely continuous invariant measure - the steady state - whose density is given by equation (\ref{eq:per0}). The fluctuation-dissipation relation follows
from a direct comparison of both equations (\ref{eq:percorr}) and (\ref{eq:per0}). There is no need to go through the interpretation of the derivative as a correlation as we did in the random environment case.

Thus let  $a=(a(x), x\in\T)$ be a smooth field of symmetric positive definite matrices defined on the unit $d$-dimensional torus $\T=\R^d/\Z^d$.
Let $\l$ be a scalar, $e_1$ be a vector in $\R^d$ and define  $\hat{\l}=\l e_1$.
Let $(X^\l_x(t)\,;\,t\ge 0)$ be the solution of the stochastic differential equation:
\beqn\label{eq:persde}
dX^\l_x(t)=b(X^\l_x(t))dt +\l a(X^\l_x(t))e_1 +\sigma(X^\l_x(t))dW_t\,;\, \ \ X^\l_x(0)=x\,,
\eeqn
where we periodically extended $a$ to $\R^d$ and defined $b=\frac 12\div a$, $\sigma=\sqrt{a}$ and $(W_t\,;\, t\ge 0)$ is a $d$-dimensional Brownian motion defined on some probability space $(\WW,\FF,P)$.

Then $(X^\l_x(t)\,;\,  t\ge0, x\in\T)$ is a Markov process with generator $\LL^\l=\frac 12  \div(a\nabla)+\hat{\l}\cdot a\nabla$. Its projection on $\T$ is a Markov process with generator
${\dot\LL}^\l=\frac 12 \div(a\nabla)+{\hat \l}\cdot a\nabla$. It admits a unique absolutely invariant measure (steady state), say $\nu_\l$, with some density $f^\l$ with respect to the Lebesgue measure on $\T$
and $f^\l$ is a solution of the equation
\beqn\label{eq:per0}
\mathrm{div}(a(\nabla f^\l-2f^\l\hat{\l}))=0\,.
\eeqn

Observe that $f^0$ is constant.

Let us now derive a first order expansion of $f^\l$ similar to what we did in Section \ref{ss_h-1}.


Given the form of equation (\ref{eq:per0}), one observes that $f^\l$ smoothly depends on $\l$. Besides the successive derivatives of $f^\l$ (as a function of $\l$) can be expressed as solutions of
the partial differential equations obtained by differentiating (\ref{eq:per0}) with respect to $\l$. Let us write $f'$ for the first derivative of $f^\l$ at $\l=0$. Using the fact that $f^0=1$, we thus get that $f'$ solves the equation
\beqn\label{eq:percorr1}
\div(a(\nabla f'-2e_1))=0\,.
\eeqn
Define $\chi_1=-\frac 12 f'$. Then (\ref{eq:percorr1}) implies that $\chi_1$ is the solution of the equation
\beqn\label{eq:percorr}
\dot{\LL}^0\chi_1=-b\cdot e_1\,.
\eeqn
Equation (\ref{eq:percorr}) is the corrector equation for the operator $\LL^0$ in the direction $e_1$, see (\ref{eq:introcorrector}).

Thus we have indeed checked that the derivative at $\l=0$ of the steady state of the operator $\LL^\l$ (symmetric diffusion operator perturbed by a constant drift of strength $\l$ in the direction $e_1$)
coincides up to multiplication by a factor $-\frac 12$ with the corrector of the symmetric drift-less operator $\LL^0$ in the direction $e_1$.

\begin{rmk} The Einstein relation, in our context,
is the equality between the so-called mobility (the derivative at $\l=0$ of the effective drift) and
the diffusion matrix for the drift-less operator $\LL^0$, see \cite{kn:GMP}.

One may observe that the Einstein relation in the periodic setting directly follows from the discussion at the beginning of this introduction.
Indeed, 
one deduces from the ergodic theorem that the process $X^\l_0$ satisfies a law of large numbers: $\frac 1 t X^\l_0(t)$ $P$-almost surely converges to the effective drift
$$\ell(\l)=\int_\T (b(\dot{x})+\l a(\dot{x})e_1)\,f^\l(\dot{x}) d\dot{x}\,,$$
and therefore
\beqnn
\frac{d}{d\l}\ell(\l)\cdot e_1\vert_{\l=0}= \int_\T (e_1\cdot b(\dot{x})f'(\dot{x})+e_1\cdot a(\dot{x})e_1\, f^0(\dot{x}))\,d\dot{x}\,.
\eeqnn
Recall that $f^0=1$. So
$$\int_\T e_1\cdot ae_1\, f^0=\int_\T e_1\cdot ae_1\,.$$

We recall that $X_0(t)\cdot e_1$ satisfies the Central Limit Theorem with asymptotic variance
\beqn\label{eq:pervar'}
\Sigma_1=\int_\T  (e_1+\nabla\chi_1(\dot{x}))\cdot a(\dot{x})(e_1+\nabla\chi_1(\dot{x}))\, d\dot{x}\,.
\eeqn
On the one hand, integration by parts, equations (\ref{eq:percorr}) and the definition of $b$ imply that
\beqnn
\frac 12 \int_\T \nabla\chi_1\cdot a\nabla\chi_1=-\int_\T({\dot\LL^0}\chi_1) \chi_1=\int_\T b\cdot e_1\, \chi_1\\
=\frac 12\int_\T \div a\cdot e_1 \,\chi_1=-\frac 12\int_\T e_1\cdot a\nabla\chi_1\,,
\eeqnn
so that (\ref{eq:pervar'}) also reads
\beqn\label{eq:pervar''}
\Sigma_1=\int_\T(e_1\cdot ae_1-\nabla\chi_1\cdot a\nabla\chi_1)\,.
\eeqn

Use the equation satisfied by $f'$, see (\ref{eq:percorr1}) and (\ref{eq:percorr}), to get that
\beqnn
\int_\T e_1\cdot b \, f'=-2\int_\T e_1\cdot b \,\chi_1 =-\int_\T e_1\cdot\div a\, \chi_1\\
=\int_\T e_1\cdot a\nabla\chi_1=-\int_\T \nabla\chi_1\cdot a\nabla\chi_1\,.
\eeqnn
Thus we obtain that
\beqn\label{eq:pereinstein}
\frac{d}{d\l}\ell(\l)\cdot e_1\vert_{\l=0}= \int_\T(-\nabla\chi_1\cdot a\nabla\chi_1+e_1\cdot ae_1)=\Sigma_1\,.
\eeqn

\end{rmk}

\begin{rmk}

Let us further comment on the main differences between the periodic and random cases.

The first difficulty one would face if trying to follow the PDE approach in the random environment setting is the necessity to solve
equation (\ref{eq:per0}) (which is now an equation in the space of environments $\Omega$). It is because we do
not even know how to make sense of equation (\ref{eq:per0}) on $\Omega$, that we used the construction of regeneration times
from section \ref{ss_reg_time}. The price we pay is {\bf Assumption 4}.

A possible approach to equation (\ref{eq:per0}) could be to try to express its solution as a power series in $\lambda$. This may not be
sufficient to solve (\ref{eq:per0}) for all values of $\lambda$ but could be good enough to get a solution for small $\lambda$'s and,
provided the power series nicely converges, one would even get the fluctuation-dissipation relation by looking at the first term of the
expansion. Indeed the linear term of the expansion should be given by the equation for the corrector. However already the equation
that the quadratic term should satisfy is problematic.

To the best of our knowledge, this `expansion approach' to construct steady states only works for dynamics satisfying a spectral gap assumption.
Then all square-integrable functions belong to $H^{-1}$ and this opens the way to iterate the corrector equation to build the different terms
of the expansion.
This approach is detailed in \cite{kn:KomOll} where the authors prove a power series expansion of the density of $\nu_\lambda$ for small $\lambda$ and
obtain some version of the Einstein relation. Under the spectral gap assumption for the un-perturbed dynamics,
the perturbed dynamics with a small but positive $\lambda$ also satisfy
the spectral gap inequality  uniformly in $\lambda$. Therefore the time it takes for the process to equilibrate stays of order
$1$ as $\lambda$ tends to $0$. This is a major difference with the situation of diffusions in a random environment as discussed in the paper at hands
where the time it takes for the perturbed process to reach equilibrium - understood as the regeneration time - is of order $\lambda^{-2}$.
In other words, the approach through regeneration times shows that the fluctuation-dissipation relations are much much more general than
what purely analytic arguments based on computations of spectral gaps and perturbation methods would a priori suggest. How general they are
is an open problem.

\end{rmk}

\section{Appendix B: alternative proof of Theorem \ref{theo:continuityIntro} }\label{sec:appendB}

In this part of the paper, we give an alternative proof of Theorem \ref{theo:continuityIntro} based on a
spectral gap argument. Recall we are assuming {\bf Assumptions 1-3}.
We use the notation from Section \ref{ss_appl}.

In the sequel, we fix an element $f$ in the space $H^{-1}_\infty(\Omega)$. To obtain an explicit bound in the next lemma,
we introduce a new norm on $H^{-1}_\infty(\Omega)$, that we denote with $\Vert f\Vert_{\bar H^{-1}_\infty(\Omega)}$ and define as
$$
\Vert f\Vert_{\bar H^{-1}_\infty(\Omega)}=\min\{\Vert \sigma^{-1}F\Vert_\infty\,;\, \div F=f\}.
$$
Clearly, due to {\bf Assumption 2}, the two norms $\Vert f\Vert_{\bar H^{-1}_\infty(\Omega)}$ and
$\Vert f\Vert_{ H^{-1}_\infty(\Omega)}$ are equivalent.

Theorem  \ref{theo:continuityIntro} follows at once from the following Lemma:

\begin{lm}\label{lm:estimateAlambdabis}
Let $p\geq 1$. Then, for all $\lambda>0$
and $t>0$, we have
\begin{equation}\label{eq:upH-1bis}
\int E\big[ \vert A^{\lambda,\omega}_{0,f}(t)\vert^p\big]\, d\mathbb Q(\omega)
\leq
(4\lambda t)^p \big( \Vert\sigma\Vert_\infty^p+\frac{2\gamma_p}{\l\sqrt{t}} \big)\Vert f\Vert_{\bar H^{-1}_\infty}^p,
\end{equation}
where $\gamma_p=\int_0^\infty p s^{p-1} e^{-\frac{s^2}2}\, ds$.
\end{lm}

\begin{rmk} \label{rm:annealedvsquenched}
Lemma \ref{lm:estimateAlambdabis} should be compared to Lemma \ref{lm:estimateAlambda}.
On the one hand, estimate (\ref{eq:upH-1}) in Lemma \ref{lm:estimateAlambda} gives a uniform
upper bound that does not depend on $\omega$, unlike (\ref{eq:upH-1bis}) where we average
with respect to $\mathbb Q$.
On the other hand, the upper bound (\ref{eq:upH-1bis}) is more explicit than (\ref{eq:upH-1}). Observe
in particular that the only way the value of $\sigma$ enters in (\ref{eq:upH-1bis}) is through the value
of $\Vert\sigma\Vert_\infty$ and, implicitely, in the definition of the norm $\Vert f\Vert_{\bar H^{-1}_\infty(\Omega)}$.
\end{rmk}

\begin{proof}

Let us derive an upper bound on the Laplace transform $\int E\big[ e^{\eta A^{\lambda,\omega}_{0,f}(t)}\big]\, d\mathbb Q(\omega)$.

Using the Girsanov transform (\ref{eq:gir}), we get that
$$ \int E\big[ e^{\eta A^{\lambda,\omega}_{0,f}(t)}\big]\, d\mathbb Q(\omega)
=
\int E\big[e^{\eta A^{0,\omega}_{0,f}(t)}e^{\lambda \bB(t)-\frac{\lambda^2}2\langle\bB\rangle(t)}\big]\, d\mathbb Q(\omega)$$
$$\leq
\sqrt{\mathbb E_0\big[e^{2\eta A_f(t)}\big]}
\sqrt{\int E\big[e^{2\lambda \bB(t)-{\lambda^2}\langle\bB\rangle(t)}\big]\, d\mathbb Q(\omega)}.$$

We have
$$ E\big[e^{2\lambda \bB(t)-{\lambda^2}\langle\bB\rangle(t)}\big]
=  E\big[e^{2\lambda \bB(t)-2{\lambda^2}\langle\bB\rangle(t)+{\lambda^2}\langle\bB\rangle(t)}\big]$$
$$\leq e^{\lambda^2 \Vert\sigma\Vert_\infty^2\, t} E\big[e^{2\lambda \bB(t)-2{\lambda^2}\langle\bB\rangle(t)}\big]
=e^{\lambda^2 \Vert\sigma\Vert_\infty^2 \, t}.$$
Therefore
\begin{equation}\label{eq:intermediatelaplace}
\int E\big[ e^{\eta A^{\lambda,\omega}_{0,f}(t)}\big]\, d\mathbb Q(\omega)
\leq \sqrt{\mathbb E_0\big[e^{2\eta A_f(t)}\big]} e^{\frac {\lambda^2}2 \Vert\sigma\Vert_\infty^2 \, t}.\end{equation}
We claim that
for all  $F\in (L^\infty(\Omega))^d$ such that $f=\div(\sigma F)$,  it holds
\begin{equation}\label{eq:laplace1}
\mathbb E_0\big[e^{2\eta A_f(t)}\big]\leq e^{8\eta^2 \Vert F\Vert_\infty^2 t}.\end{equation}

 Let us prove (\ref{eq:laplace1}). First we assume that there exists a smooth function $F\in (L^\infty(\Omega))^d$ such that $f=\div (\sigma F)$.

We use a spectral gap argument:
since we are assuming that $F$ is smooth, then $f$ is a bounded function and we have
$\vert A_f(t)\vert\leq \Vert f\Vert_\infty t$, for all $t$ and $\mathbb Q$ almost surely. In particular the Laplace transform
$\mathbb E_0\big[e^{2\eta A_f(t)}\big]$ is finite.

Let
$${\mathcal Q}_tu(\omega)=E\big[u(X_0^\omega(t).\omega)e^{2\eta A^{0,\omega}_{0,f}(t)} \big].$$
Then $({\mathcal Q}_t\,;\, t\geq 0)$ defines a strongly continuous symmetric semigroup on $L^2(\Omega)$ with Dirichlet form
$-2\eta\int fu^2\,d\mathbb Q +{\mathfrak{E}}(u,u)$.

Let
$$\Lambda(\eta)=\sup\{2\eta\int fu^2\,d\mathbb Q -{\mathfrak{E}}(u,u)\, :\, \int u^2 d\mathbb Q=1\}$$
be the largest eigenvalue of the generator of $(\mathcal Q_t\,;\, t\geq 0)$.

Then
\begin{equation}\label{eq:spectralgap}\mathbb E_0\big[e^{2\eta A_f(t)}\big]\leq e^{\Lambda(\eta)t}.\end{equation}

We now estimate $\Lambda(\eta)$.
Let $u$ be a bounded function in $\cal D$. Then $u^2$ belongs to $\cal D$ and
we have
$$\int f u^2\,d\mathbb Q=-\int 2 u \sigma Du\cdot F\, d\mathbb Q.$$
Therefore
\begin{equation}\label{eq:crucialH}
\vert \int f u^2\,d\mathbb Q\vert
\leq 2 \Vert F\Vert_\infty \sqrt{2{\mathfrak{E}}(u,u)}\sqrt{\int u^2\, d\mathbb Q}.
\end{equation}
Clearly inequality (\ref{eq:crucialH}) extends by continuity to all functions $u$ in $\cal D$.
In particular the expression $\int f u^2\, d\mathbb Q$ appearing in the definition of $\Lambda(\eta)$ defines a quadratic form on $\cal D$.

It follows from (\ref{eq:crucialH}) that, for a function $u$ in $\cal D$ such that $ \int u^2 d\mathbb Q=1$, then
$$2\eta\int fu^2\,d\mathbb Q -{\mathfrak{E}}(u,u)
\leq 4\eta\Vert F\Vert_\infty \sqrt{2{\mathfrak{E}}(u,u)}-{\mathfrak{E}}(u,u)\leq 8\eta^2 \Vert F\Vert_\infty^2.$$
So that
$$\Lambda(\eta)\leq 8\eta^2 \Vert F\Vert_\infty^2$$
and
$$
\mathbb E_0\big[e^{2\eta A_f(t)}\big]\leq e^{8\eta^2 \Vert F\Vert_\infty^2 t}.$$

This ends the proof of (\ref{eq:laplace1}) if the function $f$ is of the form $f=\div (\sigma F)$ for a smooth $F$.

Our goal now is to show that (\ref{eq:laplace1}) holds for $f$ in $H^{-1}_\infty(\Omega)$ of the form $f=\div (\sigma F)$
with an arbitrary bounded function $F\in (L^\infty(\Omega))^d$. We proceed by approximation: let $F^n$ be a sequence of smooth
functions in $(L^\infty(\Omega))^d$ that converges to $F$ in $(L^2(\Omega))^d$ and such that $\sup_n\Vert F^n\Vert_\infty\leq\Vert F\Vert_\infty$. Let $f^n=\div(\sigma F^n)$. Then the sequence $f^n$ converges to $f$ in $H^{-1}(\Omega)$.

We proved, in the discussion after Lemma \ref{l_doob}, that, for any $t>0$, then $A_{f^n}(t)$ converges towards $A_f(t)$
in $L^2(\Omega)$. We may then extract a subsequence that converges almost surely and apply Fatou's Lemma to get
(\ref{eq:laplace1}).

We conclude
from (\ref{eq:intermediatelaplace}) and (\ref{eq:laplace1})
that
$$
\int E\big[ e^{\eta A^{\lambda,\omega}_{0,f}(t)}\big]\, d\mathbb Q(\omega)
\leq e^{4\eta^2\Vert F\Vert_\infty^2t+\frac {\lambda^2}2\Vert\sigma\Vert_\infty^2\,t}.
$$
If now optimize on the choice of $F$, we obtain
\begin{equation}\label{eq:laplaceestimate}
\int E\big[ e^{\eta A^{\lambda,\omega}_{0,f}(t)}\big]\, d\mathbb Q(\omega)
\leq e^{4\eta^2\Vert f\Vert_{\bar H^{-1}_\infty}^2t+\frac {\lambda^2}2\Vert\sigma\Vert_\infty^2\,t}.
\end{equation}

We now deduce (\ref{eq:upH-1bis}) from (\ref{eq:laplaceestimate}).
To make formula more readable, we use the shorthand notation
$\Vert f\Vert=\Vert f\Vert_{\bar H^{-1}_\infty}$.

By Markov's inequality, we have
$$\int P\big[ A^{\lambda,\omega}_{0,f}(t)\geq A\lambda t\big]\, d\mathbb Q(\omega)
\leq e^{-\eta A\lambda t} \int E\big[ e^{\eta A^{\lambda,\omega}_{0,f}(t)}\big]\, d\mathbb Q(\omega).$$
By symmetry, the same holds for $-A^{\lambda,\omega}_{0,f}(t)$ and using (\ref{eq:laplaceestimate}),
we get that
$$\int P\big[\vert A^{\lambda,\omega}_{0,f}(t)\vert\geq A\lambda t\big]\, d\mathbb Q(\omega)
\leq 2 e^{-\eta A\lambda t} e^{4\eta^2\Vert f\Vert^2t+\frac {\lambda^2}2\Vert\sigma\Vert_\infty^2\,t}. $$
We choose $\eta=\frac 1 8 A\lambda \Vert f\Vert^{-1}$ and get that
$$\int P\big[\vert A^{\lambda,\omega}_{0,f}(t)\vert\geq A\lambda t\big]\, d\mathbb Q(\omega)
\leq 2 e^{-\frac 1{16} \frac {A^2\lambda^2}{\Vert f\Vert^2 }t} e^{\frac {\lambda^2}2\Vert\sigma\Vert_\infty^2\,t}. $$

Therefore
$$\int E\big[ \vert A^{\lambda,\omega}_{0,f}(t)\vert^p\big]\, d\mathbb Q(\omega)$$
$$=(\lambda t)^p \int_0^\infty p s^{p-1} \int P\big[\vert A^{\lambda,\omega}_{0,f}(t)\vert\geq s\lambda t\big]\, d\mathbb Q(\omega)\, ds$$
$$\leq (A\lambda t)^p
+ 2 e^{\frac {\lambda^2}2\Vert\sigma\Vert_\infty^2\,t} (\lambda t)^p\int_A^\infty p s^{p-1}e^{-\frac 1{16} \frac {s^2\lambda^2}{\Vert f\Vert^2 }t}$$
$$=(A\lambda t)^p
+ 2 e^{\frac {\lambda^2}2\Vert\sigma\Vert_\infty^2\,t} \big(\frac{4\Vert f\Vert}{\lambda\sqrt{t}}\big)^p
(\lambda t)^p\int_{\frac 14 A\lambda\sqrt{t}\Vert f\Vert^{-1}}^\infty p s^{p-1} e^{-s^2}\, ds$$
$$\leq (A\lambda t)^p
+ 2 e^{\frac {\lambda^2}2\Vert\sigma\Vert_\infty^2\,t} \big(\frac{4\Vert f\Vert}{\lambda\sqrt{t}}\big)^p e^{ -\frac 1{32} A^2\lambda^2t\Vert f\Vert^{-2}}
(\lambda t)^p\int_{\frac 14 A\lambda\sqrt{t}\Vert f\Vert^{-1}}^\infty p s^{p-1} e^{-\frac{s^2}2}\, ds$$
Choose $A^2=16\Vert f\Vert^2\Vert\sigma\Vert_{\infty}^2$ to get the upper bound
$$\leq (A\lambda t)^p
+ 2  \big(\frac{4\Vert f\Vert}{\lambda\sqrt{t}}\big)^p
(\lambda t)^p\int_{\frac 14 A\lambda\sqrt{t}\Vert f\Vert^{-1}}^\infty p s^{p-1} e^{-\frac{s^2}2}\, ds$$
$$ \leq (A\lambda t)^p
+ 2  \big(\frac{4\Vert f\Vert}{\lambda\sqrt{t}}\big)^p
(\lambda t)^p\int_0^\infty p s^{p-1} e^{-\frac{s^2}2}\, ds$$
$$= (4\lambda t \Vert f\Vert\Vert\sigma\Vert_\infty)^p+2\gamma_p(\lambda t)^p  \big(\frac{4\Vert f\Vert}{\lambda\sqrt{t}}\big)^p.$$

\end{proof}


\begin{thebibliography}{xxxxxx 89}





\bibitem{kn:aronson} Aronson, D. G.~(1968)\\
  Non-negative solutions of linear parabolic equations.\\
  {\em Ann.~Scuola~Norm.~Sup.~di~Pisa}~{\bf 22}~(4), 607-694.

\bibitem{kn:Baladi} Baladi, V.~(2014)\\
		Linear response, or else.\\
		 ICM Seoul, 2014, Proceedings, Vol III,  525-545.
		
\bibitem{kn:CamposRamirez} Campos, D.,  Ram\'\i rez, A.F.~(2015)\\
		Asymptotic expansion of the invariant measure for ballistic random walks in random environment in the low disorder regime. \\
		 {\em Ann.~Probab.}, to appear.	

\bibitem{kn:DFGW} De Masi, A.,  Ferrari, P., Goldstein, S.,  Wick, W.D.~(1989)\\
                  An invariance principle for reversible Markov
                  processes. Applications to random motions in random environments.\\
                 {\em J.~Stat.~Phys.}~{\bf 55}~(3/4),  787-855.

\bibitem{kn:Einstein} Einstein, A.~(1956)\\
			Investigations on the theory of the Brownian movement.\\
			Edited with notes by R. F\"urth. Translated by A. D. Cowper.
			Dover Publications, Inc., New York.
			
\bibitem{kn:FOT} Fukushima, M., Oshima, Y., Takeda, M.~(1994)\\
		{\em Dirichlet forms and symmetric Markov processes.}\\
		de Gruyter Studies in Mathematics,~{\bf 19}. Walter de Gruyter, Berlin.
		
\bibitem{kn:GMP}  Gantert, N., Mathieu, P., Piatnitski, A.~(2012)\\
 Einstein relation for reversible diffusions in a random environment.\\
   {\em Comm. Pure Appl. Math.}~{\bf 65}~(2), 187--228.

\bibitem{kn:GGN} Gantert, N., Guo, X., Nagel, J.~(2015)\\
		Einstein relation and steady states   for the random conductance model.\\
		{\em Ann.~Probab.}, to appear.	

\bibitem{kn:GT} Gilbarg, D., Trudinger, N.S.~(2001)\\
        {\em Elliptic partial differential equations of second order.}\\
         Springer-Verlag, Berlin.

\bibitem{kn:GMa} Gloria, A., Marahrens, D.\\
		Annealed estimates on the Green functions and uncertainty quantification\\
		 {\em Ann. Inst. H. Poincar\'e Anal. Non Lin\'eaire}, to appear.
		
\bibitem{kn:GO} Gloria, A., Otto, F.\\
		Quantitative results on the corrector equation in stochastic homogenization.\\
		{\em J. Eur. Math. Soc.}, to appear.
		

%

\bibitem{kn:HM} Hairer, M., Majda, A. J.~(2010)\\
A simple framework to justify linear response theory.\\
{\em Nonlinearity.}~{\bf 23}~(4), 909--922.

\bibitem{kn:JKO} Jikov, V.V., Kozlov, S.M., Oleinik, O.A.~(1994)\\
        {\em Homogenization of differential operators and integral functionals.}\\
         Springer-Verlag, Berlin.

\bibitem{kn:KV} Kipnis, C., Varadhan, S.R.S.~(1986)\\
                  A central limit theorem for additive functionals of reversible
                  Markov processes and applications to simple exclusions. \\
                 {\em Comm.~Math.~Phys.}~{\bf 104},  1-19.

\bibitem{kn:kesten} Kesten, H.~(1977)\\
	       A renewal theorem for random walk in a random environment.\\
	       {\em Symposia~Pure.~Math.}~{\bf 31}, 67-77.


\bibitem{kn:KoKr04}Komorowski, T., Krupa, G.~(2004)\\
The existence of a steady state for a perturbed symmetric random walk on a random lattice.\\
{\em Probab. Math. Statist.}  {\bf 24}(1), Acta Univ. Wratislav. No. 2646, 121–144.

\bibitem{kn:KoKr06} Komorowski, T.,  Krupa, G.~(2006)\\
The existence of a steady state for a diffusion in a random environment with an external forcing. \\
{\em  Multi scale problems and asymptotic analysis},  181–194,
GAKUTO Internat. Ser. Math. Sci. Appl., 24, Gakkōtosho, Tokyo, 2006.

\bibitem{kn:KoLaOl12} Komorowski, T., Landim, C., Olla, S.~(2012)\\
Fluctuations in Markov Processes.
Time Symmetry and Martingale Approximation.
Springer-Verlag.

\bibitem{kn:KomOll} Komorowski, T., Olla, S.~(2005)\\
On Mobility and Einstein Relation for Tracers in Time-Mixing Random Environments.\\
{\em J. Stat Phys.}~{\bf 118}, 407-435.

\bibitem{kn:Ko80} Kozlov, S.M.~(1980)\\
                Averaging of Random Operators\\
                {\em Sb. Math.}~{\bf 37}~(2), 167-180.

\bibitem{kn:Ko85} Kozlov, S.M.~(1985)\\
                The method of averaging and walks in inhomogeneous environments\\
                {\em Russian~Math.~Surveys}~{\bf 40}~(2), 73-145.

\bibitem{kn:Ku85} Kubo, R., Toda, M., Hashitsume, N. (1985)\\
 {\em Statistical physics. II. Nonequilibrium statistical mechanics.}\\
 Springer Series in Solid-State Sciences, 31. Springer-Verlag, Berlin.

\bibitem{kn:LR} Lebowitz, J.L., Rost, H.~(1994)\\
	       The Einstein relation for the displacement of a test particle in a random environment.\\
	       {\em Stochastic~Process.~Appl.}~{\bf 54}~(2), 183-196.


\bibitem{kn:Mourrat} Mourrat, J-C.~(2001)\\
		Variance decay for functionals of the environment viewed by the particle. \\
		{\em Ann.~I.~H.~Poincar\'e-PR}~{\bf 47}~(1), 294-327.

\bibitem{kn:Osada} Osada, H.~(1982)\\
	       Homogenization of diffusion processes with random stationary coefficients.\\
	       {\em  Probability theory and mathematical statistics} (Tbilisi, 1982),  507--517, Lecture Notes in Math., 1021,
		Springer, Berlin, 1983.
	
\bibitem{kn:PV} Papanicolaou, G., Varadhan, S.R.S.~(1982)\\
	        Diffusion with random coefficients.\\
	       {\em Essays in honor of C.R. Rao, ed. by G. Kallianpur, P.R. Krishnajah, J.K. Gosh}
	       pp. 547--552. Amsterdam: North Holland 1982.


\bibitem{kn:Perrin} Perrin, J.~(1913)\\
		Les atomes.\\
	First edited in 1913 by F\'elix Alcan. Reissued in 2014 by CNRS Editions.

\bibitem{kn:RYor} Revuz, D., Yor, M.~(1994)\\
		{\em  Continuous martingales and Brownian motion.}\\
		Grundlehren der Mathematischen Wissenschaften\\
		(Fundamental Principles of Mathematical Sciences)~{\bf 293}.
		Springer, Berlin.

\bibitem{kn:LS} Shen, L.~(2003)\\
                On ballistic diffusions in random environment\\
                {\em Ann.~I.~H.~Poincar\'e-PR}~{\bf 39}~(5), 839-876.

\bibitem{kn:SZ} Sznitman, A., Zerner, M.~(1999)\\
		A law of large numbers for random walks in random environment.\\
		{\em Ann. Probab.}~{\bf  27}, 1851--1869.



\end{thebibliography}
\end{document}